\documentclass[11pt]{amsart}

\usepackage{amsmath,amsthm,color}

\usepackage{amssymb}

\usepackage{pdfsync}

\newcommand{\C}{{\mathbb C}}       
\newcommand{\R}{{\mathbb R}}       
\newcommand{\N}{{\mathbb N}}       %

\newcommand{\DD}{{\mathcal D}}
\newcommand{\HH}{{\mathcal H}}

\newcommand{\BZ}{{\mathcal B}}

\newcommand{\WW}{{\mathcal W}}

\newcommand{\TT}{{\mathcal T}}
\newcommand{\MD}{{\mathcal {MD}}}

\newcommand{\cH}{{\mathcal{H}}}

\newcommand{\diam}{{\rm diam}}
\newcommand{\dist}{{\rm dist}}

\newcommand{\fiproof}{{\hspace*{\fill} $\square$ \vspace{2pt}}}

\newcommand{\rf}[1]{{(\ref{#1})}}
\newcommand{\supp}{\operatorname{supp}}

\newcommand{\ve}{{\varepsilon}}
\newcommand{\vv}{{\vspace{2mm}}}
\newcommand{\vvv}{{\vspace{3mm}}}
\newcommand{\wt}[1]{{\widetilde{#1}}}
\newcommand{\wh}[1]{{\widehat{#1}}}

\newcommand{\LD}{{\mathsf{LD}}}
\newcommand{\HD}{{\mathsf{HD}}}
\newcommand{\UB}{{\mathsf{UB}}}

\newcommand{\sss}{{\mathsf{Stop}}}
\newcommand{\nex}{{\mathsf{Next}}}

\newcommand{\ttt}{{\mathsf{Top}}}

\newcommand{\reg}{{\mathsf{Reg}}}

\newcommand{\tree}{{\mathsf{Tree}}}
\newcommand{\tr}{{\mathsf{Tr}}}
\newcommand{\treg}{{\mathsf{Treg}}}

\newcommand{\eend}{{\mathsf{End}}}

\newtheorem{theorem}{Theorem}[section]
\newtheorem*{theorem*}{Theorem}
\newtheorem*{lemma*}{Lemma}
\newtheorem*{theorema*}{Theorem A}
\newtheorem*{theoremb*}{Theorem B}
\newtheorem*{theoremc*}{Theorem C}
\newtheorem*{mainlemma*}{Main Lemma}
\newtheorem{mainlemma}[theorem]{Main Lemma}
\newtheorem{lemma}[theorem]{Lemma}

\newtheorem{coro}[theorem]{Corollary}

\theoremstyle{definition}

\newcommand{\ps}[1]{\left( #1 \right)}

\newcommand{\av}[1]{\left| #1 \right|}
\newcommand{\ip}[1]{\left\langle #1 \right\rangle}

\def\Xint#1{\mathchoice		
{\XXint\displaystyle\textstyle{#1}}%
{\XXint\textstyle\scriptstyle{#1}}%
{\XXint\scriptstyle\scriptscriptstyle{#1}}%
{\XXint\scriptscriptstyle\scriptscriptstyle{#1}}%
\!\int}
\def\XXint#1#2#3{{\setbox0=\hbox{$#1{#2#3}{\int}$ }
\vcenter{\hbox{$#2#3$ }}\kern-.58\wd0}}

\def\avint{\Xint-}

\theoremstyle{remark}
\newtheorem{remark}[theorem]{\bf Remark}

\numberwithin{equation}{section}

\begin{document}

\title[Rectifiability in terms of Jones' square function: Part II]{Characterization of 
$n$-rectifiability in terms of Jones' square function: Part II
}

\author{Jonas Azzam}
\address{Departament de Matem\`atiques\\ Universitat Aut\`onoma de Barcelona \\ Edifici C Facultat de Ci\`encies\\08193 Bellaterra (Barcelona, Catalonia) }
\email{jazzam@mat.uab.cat}

\author{Xavier Tolsa}
\address{ICREA and Departament de Matem\`atiques\\ Universitat Aut\`onoma de Barcelona \\ Edifici C Facultat de Ci\`encies\\
08193 Bellaterra (Barcelona, Catalonia) }
\email{xtolsa@mat.uab.cat}

\thanks{J.A. and X.T. were supported by the ERC grant 320501 of the European Research Council (FP7/2007-2013).
X.T. was also partially supported by the grants 2014-SGR-75 (Catalonia) and MTM2013-44304-P (Spain).}

\maketitle

\begin{abstract}
We show that a Radon measure $\mu$ in $\R^d$ which is absolutely continuous with respect to the $n$-dimensional Hausdorff measure $\HH^n$ is $n$-rectifiable if the so called Jones' square function is finite $\mu$-almost everywhere. The converse of this result is proven in a companion paper by the second author, and hence these two results give a classification of all $n$-rectifiable measures which are absolutely continuous with respect to $\HH^{n}$. Further, in this paper we also investigate the relationship between the Jones' square function and the so called Menger curvature of a measure with linear growth, and
we show an application to the study of analytic capacity. 
\end{abstract}

\section{Introduction}

Let $\mu$ be a Radon measure in $\R^d$.
One says that $\mu$ is $n$-rectifiable if
there are Lipschitz maps
$f_i:\R^n\to\R^d$, $i=1,2,\ldots$, such that 
\begin{equation}\label{eqdef00}
\mu\biggl(\R^d\setminus\bigcup_i f_i(\R^n)\biggr) = 0,
\end{equation}
and $\mu$ is absolutely continuous with respect to the $n$-dimensional Hausdorff measure $\HH^n$. 
A set $E\subset \R^d$ is called $n$-rectifiable if the measure $\HH^n|_E$ is $n$-rectifiable.

This is the second of a series of two papers where we contribute to characterize when measures in the Euclidean space are rectifiable. This subject stems from the work of Besicovitch (\cite{Bes1}, \cite{Bes2}, \cite{Bes3}) who first discovered the geometric dichotomy between $1$-rectifiable sets (sets that may be covered by one-dimensional Lipschitz graphs) and purely $1$-unrectifiable sets (sets that have $\HH^{1}$-measure zero intersection with any Lipschitz graph). Finding criteria to distinguish these two classes of sets has become a field of its own due to its applications to various analytic fields.

A particular example we mention is the one of singular integrals. In \cite{DS1}, David and Semmes studied $n$-Ahlfors-David regular (or $n$-AD) measures $\mu$ (meaning the measure of any ball centred on its support has $\mu$-measure at least and at most a constant times the radius of the ball to the power $n$) and classified which of these measures are {\it uniformly rectifiable}. There are several equivalent definitions of this term: for David and Semmes, the uniformly rectifiable measures $\mu$ were firstly the $n$-AD regular measures for which a certain big class of singular integral operators with odd kernel is  bounded in $L^2(\mu)$; the geometric definition, however, is more transparent, and says that any ball $B$ centered on $\supp\mu$ with radius $r(B)$ contains an $L$-Lipschitz image of a subset of $\mathbb{R}^{n}$ of measure at least $c\,r(B)^{n}$, where $c$ and $L$ are fixed constants. 

Over the course of \cite{DS1} and \cite{DS2}, David and Semmes derived several other equivalent formulations of uniform wrectifiability. We will review the one that most concerns us presently: Given a closed ball $B \subset \R^d$ with radius $r(B)$, an integer $0<n<d$, and $1\leq p <\infty$, let
$$\beta_{\mu,p}^n(B) = \inf_L \left(\frac1{r(B)^n} \int_{B} \left(\frac{\dist(y,L)}{r(B)}\right)^p\,d\mu(y)\right)^{1/p},$$
where the infimum is taken over all $n$-planes $L\subset \R^d$. Given a fixed $n$, to simplify notation we will drop the exponent $n$ and we will write $\beta_{\mu,p}(B)$ instead of $\beta_{\mu,p}^n(B)$. Then an $n$-AD regular measure $\mu$ is uniformly rectifiable if and only if there is some $c>0$ so that, for any $B$ centred on $\supp \mu$,
\[\int_{B}\int_{0}^{r(B)}\beta_{\mu,2}(B(x,r))^2\frac{dr}{r}d\mu(x)\leq c\mu(B).\]

The $\beta_{\mu,p}$ coefficients are a generalization of the so-called Jones $\beta$-numbers introduced in \cite{J-TSP}; there, he used an $L^{\infty}$-version of $\beta_{\mu,p}$ to characterize {\it all} compact subsets of the plane which can be contained in a rectifiable set, a characterization that extends to higher dimensions and even to Hilbert spaces (see \cite{O-TSP} and \cite{S-TSP}).

Many of the conditions in the David and Semmes theory rely heavily on the AD regularity assumption on the measure. In this paper and its prequel, we show that a suitable version of the above characterization just mentioned extends to much more general Radon measures.

The upper and lower $n$-dimensional densities of $\mu$ at a point $x\in\R^d$ are defined, respectively, by
$$\Theta^{n,*}(x,\mu)= \limsup_{r\to 0}\frac{\mu(B(x,r))}{(2r)^n},\qquad
\Theta^{n}_*(x,\mu)= \liminf_{r\to 0}\frac{\mu(B(x,r))}{(2r)^n}.$$
In case both coincide, we denote $\Theta^{n}(x,\mu)=\Theta^{n,*}(x,\mu)=\Theta^{n}_*(x,\mu)$, and this is called
the $n$-dimensional density of $\mu$ at $x$.

The main result of this paper is the following:

\begin{theorem}\label{teo1}
Let $\mu$ be a finite Borel measure in $\R^d$ 
such that $0<\Theta^{n,*}(x,\mu)<\infty$ for $\mu$-a.e.\ $x\in\R^d$. If
\begin{equation}\label{eqjones*}
\int_0^1 \beta_{\mu,2}(x,r)^2\,\frac{dr}r<\infty \quad\mbox{ for $\mu$-a.e.\ $x\in\R^d$,}
\end{equation}
then $\mu$ is $n$-rectifiable.
\end{theorem}
\vv

The integral on the left side of \rf{eqjones*} is a version of the so called Jones' square function. 
In the part I paper \cite{Tolsa-nec} it has been shown that this is finite $\mu$-a.e.\ if $\mu$ is $n$-rectifiable. 
So by  combining this result with Theorem \ref{teo1} we deduce the following result:

\begin{coro}\label{coro1}
Let $\mu$ be a finite Borel measure in $\R^d$ 
such that $0<\Theta^{n,*}(x,\mu)<\infty$ for $\mu$-a.e.\ $x\in\R^d$.  Then $\mu$ is $n$-rectifiable if and only if
$$\int_0^1 \beta_{\mu,2}(x,r)^2\,\frac{dr}r<\infty \quad\mbox{ for $\mu$-a.e.\ $x\in\R^d$.}$$
\end{coro}
\vv

Notice that the preceding corollary applies to any measure $\mu=\HH^n|_E$, with $\HH^n(E)<\infty$. So we have:

\begin{coro}\label{coro2}
Let $E\subset\R^d$ be an $\HH^n$-measurable set with $\HH^n(E)<\infty$. The set $E$ is $n$-rectifiable if and only if
$$\int_0^1 \beta_{\HH^n|_E,2}(x,r)^2\,\frac{dr}r<\infty \quad\mbox{ for $\HH^n$-a.e.\ $x\in E$.}$$
\end{coro}
\vv

According to \cite{Badger-Schul}, Peter Jones conjectured in 2000 that, given an arbitrary Radon measure $\mu$ in $\R^d$, some condition in the spirit of \rf{eqjones*} should imply that $\mu$ is $n$-rectifiable in the sense that there are Lipschitz maps
$f_i:\R^n\to\R^d$, $i=1,2,\ldots$, such that \rf{eqdef00} holds, without assuming $\mu$ to be absolutely continuous with respect 
to Hausdorff measure.
Corollary \ref{coro1} shows that this conjecture holds precisely in the particular case when $\mu\ll\HH^n$. 


Let us remark that Theorem \ref{teo1} was already known to hold under the additional assumption that the lower density
$\Theta_*^n(x,\mu)$ is positive $\mu$-a.e., by a theorem due to Pajot \cite{Pa} valid for $\mu=\HH^{n}|_{E}$, recently extended by Badger and Schul \cite{BS2} to Radon measures such that $0<\Theta_*^n(x,\mu)\leq \Theta^{n,*}(x,\mu)<\infty$ $\mu$-a.e.
The hypothesis on the positiveness of the lower density 
$\Theta_*^n(x,\mu)$ is essential in the arguments in \cite{Pa} and \cite{BS2} because, roughly speaking, it allows the authors to reduce their assumptions to the case where
the measure $\mu$ is supposed to be $n$-AD-regular. 

Recall that if $\mu$ is a measure of the form $\mu=\HH^n|_E$, with $\HH^n(E)<\infty$, then
$\Theta^{n,*}(x,\mu)>0$ $\mu$-a.e., while it may happen that $\Theta_*^{n}(x,\mu)=0$ $\mu$-a.e. So from Pajot's theorem and
its further generalization by Badger and Schul one cannot deduce the characterization of $n$-rectifiable sets in Corollary
\ref{coro2}.


We will prove Theorem \ref{teo1} by means of a suitable corona type decomposition in terms of some ``dyadic cubes''
introduced by David and Mattila \cite{David-Mattila}. This corona type decomposition has some similarities with the one
from \cite{Tolsa-delta}: it splits the dyadic lattice into some collections of cubes, which we will call ``trees'', where,
roughly speaking, the density of $\mu$ does not oscillate too much and most of the measure is concentrated close to a Lipschitz manifold. To construct this Lipschitz manifold we will use a nice theorem from David and Toro \cite{DT1} which is appropriate
to parametrize Reifenberg flat sets with holes and is particularly well adapted to constructions of Lipschitz 
manifolds involving stopping time arguments, such as in our case. Further, we will show that the family of trees of the corona decomposition
satisfies a packing condition by following arguments inspired by the some of the techniques used in \cite{Tolsa-bilip}
to prove the bilipschitz ``invariance'' of analytic capacity.
\vv

Our second main result in this paper relates the curvature of a measure to its $\beta_{\mu,2}$-numbers, and it involves also the densities $\Theta_\mu^1(x,r)=\dfrac{\mu(B(x,r))}{r}$, for $x\in\R^d$ and $r>0$.
Recall that
the so called {\it Menger curvature} of $\mu$ is defined by 
$$c^2(\mu) = \iiint\frac1{R(x,y,z)^2}\,d\mu(x)\,d\mu(y)\,d\mu(z),$$
where $R(x,y,z)$ stands for the radius of the circumference passing through $x,y,z$. The notion of
curvature of a measure was introduced by Melnikov \cite{Melnikov} while studying analytic capacity.
Because its relationship to the Cauchy transform on the one hand and to $1$-rectifiability on the other hand
(see \cite{MV} and \cite{Leger}),
curvature of measures has played a key role in the solution of Vitushkin's conjecture by David \cite{David-vitus} and in the
in the proof of the semiadditivity of analytic capacity by Tolsa \cite{Tolsa-sem}. Here we prove the following:

\vv
\begin{theorem}  \label{teocurv-intro}
Let $n=1$, $d=2$, $\mu$ be a finite compactly supported Radon measure in $\R^2$  such that $\mu(B(x,r))\leq r$ for all $x\in\R^2$ and $r>0$.
Then
$$c^2(\mu) + \|\mu\|\sim \iint_0^\infty \beta_{\mu,2}(x,r)^2\,\Theta_\mu^1(x,r)\,\frac{dr}r\,d\mu(x) + \|\mu\|,$$
where the implicit constant is an absolute constant.
\end{theorem}

\vv
The notation $A\sim B$ means that
there is some fixed positive constant $c$ such that $c^{-1}A\leq B\leq c\,A$.

If $\mu$ is $1$-AD-regular, the result stated in Theorem \ref{teocurv-intro} was already known. In fact,  in this case 
one has the more precise estimate
$$c^2(\mu) \sim \iint_0^\infty \beta_{\mu,2}(x,r)^2\,\frac{dr}r\,d\mu(x).$$
That the double integral on the right hand side does not exceed $c^2(\mu)$ times some constant only depending on the $1$-AD-regularity of $\mu$ was proved by
Mattila, Melnikov and Verdera in \cite{MMV}. The converse inequality is essentially due to Peter Jones. The reader can find the proof of both estimates in Chapters 3 and 7 of \cite{Tolsa-llibre}, for example. For other analogous results for some higher dimensional versions of curvature, see the works \cite{LW1} and \cite{LW2} of Lerman and Whitehouse.
\vv

Recall that the analytic capacity of a compact subset $E\subset\C$ is defined by 
$$\gamma(E) =\sup|f'(\infty)|,$$
where the supremum is taken over all analytic functions $f:\C\setminus E\to \C$ with $\|f\|_{\infty}\leq1$ and
$f'(\infty)=\lim_{z\to \infty} z(f(z)-f(\infty)).$

The notion of analytic capacity was introduced by Ahlfors \cite{Ahlfors} in order to study the so called Pianlev\'e problem, which consists in characterizing the removable sets for bounded analytic functions in metric and geometric terms. He showed that a compact set $E\subset \C$ is
removable if and only if $\gamma(E)=0$, and thus he reduced the Painlev\'e problem to the metric-geometric characterization of sets with vanishing analytic capacity. From the description of analytic capacity in terms of curvature obtained in 
\cite{Tolsa-sem} and Theorem \ref{teocurv-intro} we get the following.

\begin{coro}\label{corogam}
Let $E\subset\C$ be compact. Then
$$\gamma(E)\sim\supp\mu(E),$$
where the supremum is taken over all Borel measures $\mu$ in $\C$ such that
$$\sup_{r>0}\Theta_\mu^1(x,r) +
\int_0^\infty \beta_{\mu,2}(x,r)^2\,\Theta_\mu^1(x,r)\,\frac{dr}r\leq1.$$ 
\end{coro}

\vv

The plan of the paper is the following. In Section \ref{secdyad} we recall the properties of the dyadic lattice of David and
Mattila mentioned above. In Section \ref{secbalan} we introduce the so called balanced balls and we prove
some technical results about them which will be necessary for the construction of the corona decomposition.
We state the Main Lemma \ref{mainlemma} in Section \ref{secmlemma}. This is the essential ingredient for the construction
of the corona decomposition in the subsequent Section \ref{secprovam}. Theorem \ref{teo1} is proved in the
same Section \ref{secprovam} by using this corona decomposition. On the other hand, the proof of Main Lemma \ref{mainlemma} is deferred to Section \ref{secfi}. 
In Section \ref{secrem} we explain that the 
assumption \rf{eqjones*} in Theorem \ref{teo1} can be weakened by multiplying the integrand $\beta_{2,\mu}(x,r)^2$ by any power of the density 
$\mu(B(x,r))/r^n$, which we then use to prove Theorem \ref{teocurv-intro}. This proof will rely heavily on the proofs and results contained in \cite{Tolsa-bilip} and \cite{Tolsa-delta}, and we recommend reading this section with these papers on hand as references.

\vv

In this paper  we will use the letters $c,C$ to denote
absolute constants which may change their values at different
occurrences. On the other hand, constants with subscripts, such as $c_1$, do not change their values
at different occurrences.

The notation $A\lesssim B$ means that
there is some fixed constant $c$ such that $A\leq c\,B$. So $A\sim B$ is equivalent to $A\lesssim B\lesssim A$. 
If we want to write explicitely 
the dependence on some constants $c_1$ of the relationship such as ``$\lesssim$'', we will write
 $A\lesssim_{c_1} B$. \\


\section{The dyadic lattice of cubes with small boundaries}\label{secdyad}
\vv

We will use the dyadic lattice of cubes
with small boundaries constructed by David and Mattila in \cite[Theorem 3.2]{David-Mattila}.
The properties of this dyadic lattice are summarized in the next lemma.
\vv

\begin{lemma}[David, Mattila]
\label{lemcubs}
Let $\mu$ be a Radon measure on $\R^d$, $E=\supp\mu$, and consider two constants $C_0>1$ and $A_0>5000\,C_0$. Then there exists a sequence of partitions of $E$ into
Borel subsets $Q$, $Q\in \DD_k$, with the following properties:
\begin{itemize}
\item For each integer $k\geq0$, $E$ is the disjoint union of the ``cubes'' $Q$, $Q\in\DD_k$, and
if $k<l$, $Q\in\DD_l$, and $R\in\DD_k$, then either $Q\cap R=\varnothing$ or else $Q\subset R$.
\vv

\item The general position of the cubes $Q$ can be described as follows. For each $k\geq0$ and each cube $Q\in\DD_k$, there is a ball $B(Q)=B(z_Q,r(Q))$ such that
$$z_Q\in E, \qquad A_0^{-k}\leq r(Q)\leq C_0\,A_0^{-k},$$
$$E\cap B(Q)\subset Q\subset E\cap 28\,B(Q)=E \cap B(z_Q,28r(Q)),$$
and
$$\mbox{the balls $5B(Q)$, $Q\in\DD_k$, are disjoint.}$$

\vv
\item The cubes $Q\in\DD_k$ have small boundaries. That is, for each $Q\in\DD_k$ and each
integer $l\geq0$, set
$$N_l^{ext}(Q)= \{x\in E\setminus Q:\,\dist(x,Q)< A_0^{-k-l}\},$$
$$N_l^{int}(Q)= \{x\in Q:\,\dist(x,E\setminus Q)< A_0^{-k-l}\},$$
and
$$N_l(Q)= N_l^{ext}(Q) \cup N_l^{int}(Q).$$
Then
\begin{equation}\label{eqsmb2}
\mu(N_l(Q))\leq (C^{-1}C_0^{-3d-1}A_0)^{-l}\,\mu(90B(Q)).
\end{equation}
\vv

\item Denote by $\DD_k^{db}$ the family of cubes $Q\in\DD_k$ for which
\begin{equation}\label{eqdob22}
\mu(100B(Q))\leq C_0\,\mu(B(Q)).
\end{equation}
We have that $r(Q)=A_0^{-k}$ when $Q\in\DD_k\setminus \DD_k^{db}$
and
\begin{equation}\label{eqdob23}
\mu(100B(Q))\leq C_0^{-l}\,\mu(100^{l+1}B(Q))\quad
\mbox{for all $l\geq1$ such that $100^l\leq C_0$ and $Q\in\DD_k\setminus \DD_k^{db}$.}
\end{equation}
\end{itemize}
\end{lemma}

\vv
We use the notation $\DD=\bigcup_{k\geq0}\DD_k$. For $Q\in\DD$, we set $\DD(Q) =
\{P\in\DD:P\subset Q\}$.
Given $Q\in\DD_k$, we denote $J(Q)=k$. We set
$\ell(Q)= 56\,C_0\,A_0^{-k}$ and we call it the side length of $Q$. Note that 
$$\frac1{28}\,C_0^{-1}\ell(Q)\leq \diam(Q)\leq\ell(Q).$$
Observe that $r(Q)\sim\diam(Q)\sim\ell(Q)$.
Also we call $z_Q$ the center of $Q$, and the cube $Q'\in \DD_{k-1}$ such that $Q'\supset Q$ the parent of $Q$.
 We set
$B_Q=28 \,B(Q)=B(z_Q,28\,r(Q))$, so that 
$$E\cap \tfrac1{28}B_Q\subset Q\subset B_Q.$$

We assume $A_0$ big enough so that the constant $C^{-1}C_0^{-3d-1}A_0$ in 
\rf{eqsmb2} satisfies 
$$C^{-1}C_0^{-3d-1}A_0>A_0^{1/2}>10.$$
Then we deduce that, for all $0<\lambda\leq1$,
\begin{equation}\label{eqfk490}
\mu\bigl(\{x\in Q:\dist(x,E\setminus Q)\leq \lambda\,\ell(Q)\}\bigr) + 
\mu\bigl(\bigl\{x\in 4B_Q:\dist(x,Q)\leq \lambda\,\ell(Q)\}\bigr)\leq
c\,\lambda^{1/2}\,\mu(3.5B_Q).
\end{equation}

We denote
$\DD^{db}=\bigcup_{k\geq0}\DD_k^{db}$ and $\DD^{db}(Q) = \DD^{db}\cap \DD(Q)$.
Note that, in particular, from \rf{eqdob22} it follows that
$$\mu(100B(Q))\leq C_0\,\mu(Q)\qquad\mbox{if $Q\in\DD^{db}.$}$$
For this reason we will call the cubes from $\DD^{db}$ doubling. 

As shown in \cite[Lemma 5.28]{David-Mattila}, any cube $R\in\DD$ can be covered $\mu$-a.e.\
by a family of doubling cubes:
\vv

\begin{lemma}\label{lemcobdob}
Let $R\in\DD$. Suppose that the constants $A_0$ and $C_0$ in Lemma \ref{lemcubs} are
chosen suitably. Then there exists a family of
doubling cubes $\{Q_i\}_{i\in I}\subset \DD^{db}$, with
$Q_i\subset R$ for all $i$, such that their union covers $\mu$-almost all $R$.
\end{lemma}

The following result is proved in \cite[Lemma 5.31]{David-Mattila}.
\vv

\begin{lemma}\label{lemcad22}
Let $R\in\DD$ and let $Q\subset R$ be a cube such that all the intermediate cubes $S$,
$Q\subsetneq S\subsetneq R$ are non-doubling (i.e.\ belong to $\bigcup_{k\geq0}\DD_k\setminus \DD_k^{db}$).
Then
\begin{equation}\label{eqdk88}
\mu(100B(Q))\leq A_0^{-10n(J(Q)-J(R)-1)}\mu(100B(R)).
\end{equation}
\end{lemma}

Let us remark that the constant $10$ in \rf{eqdk88} can be replaced by any other positive 
constant if $A_0$ and $C_0$ are chosen suitably in Lemma \ref{lemcubs}, as shown in (5.30) of
\cite{David-Mattila}.

Given a ball $B\subset \R^d$, we consider its $n$-dimensional density:
$$\Theta_\mu(B)= \frac{\mu(B)}{r(B)^n}.$$
We will also write $\Theta_\mu(x,r)$ instead of $\Theta_\mu(B(x,r))$.

From the preceding lemma we deduce:

\vv
\begin{lemma}\label{lemcad23}
Let $Q,R\in\DD$ be as in Lemma \ref{lemcad23}.
Then
$$\Theta_\mu(100B(Q))\leq C_0\,A_0^{-9n(J(Q)-J(R)-1)}\,\Theta_\mu(100B(R))$$
and
$$\sum_{S\in\DD:Q\subset S\subset R}\Theta_\mu(100B(S))\leq c\,\Theta_\mu(100B(R)),$$
with $c$ depending on $C_0$ and $A_0$.
\end{lemma}

\begin{proof}
By \rf{eqdk88},
\begin{align*}
\Theta_\mu(100B(Q))&\leq A_0^{-10n(J(Q)-J(R)-1)}\,\frac{\mu(100B(R))}{r(100B(Q))^n} \\& = 
A_0^{-10n(J(Q)-J(R)-1)}\Theta_\mu(100B(R))\,\frac{r(B(R))^n}{r(B(Q))^n}.
\end{align*}
The first inequality in the lemma follows from this estimate and the fact that
$r(B(R))\leq C_0\,A_0^{(J(Q)-J(R))}\,r(B(Q))$.

The second inequality in the lemma is an immediate consequence of the first one.
\end{proof}

\vv
From now on we will assume that $C_0$ and $A_0$ are some big fixed constants so that the
results stated in the lemmas of this section hold.

\vv


\section{Balanced balls}\label{secbalan}
\vv

\begin{lemma}\label{lembalan}
Let $\mu$ be a Radon measure in $\R^d$, and let $B\subset \R^d$ be some ball with radius $r>0$ such that $\mu(B)>0$.
Let $0<t<1$ and $0<\gamma<1$. Then there exists some constant $\ve=\ve(t)>0$ such that one of following alternatives holds:
\begin{itemize}
\item[(a)] There are points $x_0,x_1,\ldots,x_n\in B$ such that 
$$\mu(B(x_k,tr)\cap B)\geq \ve\,\mu(B) \qquad\mbox{for $0\leq k\leq n$,}$$
and if $L_k$ stands for the $k$-plane that passes through $x_0,x_1,\ldots,x_k$, then
$$\dist(x_k,L_{k-1}) \geq \gamma\,r\qquad\mbox{for $1\leq k\leq n$.}$$

\item[(b)] There exists a family of balls $\{B_i\}_{i\in I_B}$, with radii $r(B_i)=4\gamma r$, centered on $B$, so that
the balls $\{10B_i\}_{i\in I_B}$ are pairwise disjoint, 
$$\sum_{i\in I_B}\mu(B_i)\gtrsim \mu(B),$$
and 
$$\Theta_\mu(B_i)\gtrsim \gamma^{-1}\,\Theta_\mu(B)\quad \mbox{ for all $i\in I_B$}.$$
\end{itemize}
\end{lemma}

Note that the constant $\ve$ above depends on $t$ but not on $\gamma$.

\begin{proof}
 We choose $x_0,x_1,\ldots,x_n$ inductively as follows. First we take $x_0\in B$ such
$$\mu(B(x_0,tr))\geq \frac12 \,s_0,$$
where 
$$s_0= \sup_{x\in B} \mu(B(x,tr)).$$
If $x_0,\ldots, x_{k-1}$ have been chosen, we take $x_k\in B$ such that $\dist(x_k,L_{k-1}) \geq\gamma \,r$ (where 
$L_{k-1}$ is a plane that passes through $x_0,\ldots,x_{k-1}$) and
$$\mu(B(x_k,tr))\geq \frac12 \,s_k,$$
where
$$s_k= \sup_{x\in B:\dist(x,L_{k-1})\geq\gamma r} \mu(B(x,tr)).$$

If $s_k\geq 2\ve\,\mu(B)$ for all $0\leq k\leq n$ (where $\ve$ will be fixed below), then the alternative (a) in the lemma holds.
Otherwise, let $k_0$ be such that $s_{k_0}< 2\ve\,\mu(B)$, with $k_0$ minimal. Notice first that
$$s_0\geq c\, t^{d}\,\mu(B),$$
since $B$ can be covered by at most $c\,t^{-d}$ balls of radii $t\,r$. Thus, assuming $\ve < c\,t^{d}/2$, we get $s_0\geq 2\ve\,\mu(B)$ and
thus $k_0\geq1$. In this way, the fact that $s_{k_0}< 2\ve\,\mu(B)$ means that any ball $B(x,t\,r)$ with $x\in B\setminus U_{\gamma r}
(L_{k_0-1})$ (where $U_h(A)$ stands for the $h$-neighborhood of $A$) satisfies
$$\mu(B(x,t\,r))<\ve\,\mu(B).$$
Since $B\setminus U_{\gamma r}(L_{k_0-1})$ can be covered by $c\,t^{-d}$ balls of this type, we infer that
$$\mu(B\setminus U_{\gamma r}(L_{k_0-1}))\leq c\,t^{-d}\,\ve\,\mu(B).$$
As a consequence, if $\ve$ is small enough (i.e.\ $\ve\ll t^{d}$), it turns out that 
\begin{equation}\label{eqajgd23}
\mu(B\cap U_{\gamma r}(L_{k_0-1}))\geq \frac12\,\mu(B).
\end{equation}

It is easy to check that $B\cap U_{\gamma r}(L_{k_0-1})$ can be covered by at most $N$ balls $B_i$ of radii $4\gamma r$, with
$$N=: c\,\frac{(\gamma r)^{d-(k_0-1)}\,r^{k_0-1}}{(\gamma\,r)^d} = c\,\gamma^{-(k_0-1)}\leq c\, \gamma^{-(n-1)}.$$
Let $\wt I_B$ the subfamily of the balls $B_i$ such that 
$\mu(B_i)\geq \dfrac{\mu(B)}{4N}$.
Since
$$\mu\biggl(\bigcup_{i\not \in \wt I_B} B_i\biggr) \leq \sum_{i\not \in \wt I_B} \mu(B_i)\leq N\,\frac{\mu(B)}{4N} = \frac14
\,\mu(B),$$
from \rf{eqajgd23} we infer that
$$\mu\biggl(\bigcup_{i\in \wt I_B} B_i\biggr) \geq \frac12\,\mu(B)- \mu\biggl(\bigcup_{i\not \in \wt I_B} B_i\biggr)\geq \frac14
\,\mu(B).$$

Further, by an elementary covering argument, it follows that there exists a subfamily $B_i$, $i\in I_B \subset\wt I_B$, such that the balls
$10\,B_i$, $i\in I_B$, are parwise disjoint and 
$$\mu\biggl(\,\bigcup_{i\in I_B} B_i\biggr) \gtrsim \mu\biggl(\,\bigcup_{i\in \wt I_B} B_i\biggr)\gtrsim \mu(B)$$
On the other hand, since $\mu(B_i)\geq \dfrac{\mu(B)}{4N}$ for each $i\in I_B$ and $N\lesssim \gamma^{-(n-1)}$, we deduce that
$$\Theta_\mu(B_i)\geq \frac{\mu(B)}{(4N)(4\gamma r)^n} \gtrsim \frac{\mu(B)}{\gamma r^n} = \gamma^{-1}\,\Theta_\mu(B)
\quad \mbox{ for all $i\in I_B.$}$$
Thus we have shown that the alternative (b) holds.
\end{proof}

\vv

For a fixed Radon measure $\mu$ in $\R^d$, let $B\subset \R^d$ be a ball with $\mu(B)>0$. 
We say that $B$ is $(t,\gamma)$-balanced  if the alternative (a) in Lemma \ref{lembalan} holds
with parameters $t>0$, $\ve(t)$, and $\gamma>0$. 

\begin{remark} \label{remt0g}
Notice that, for a given $\gamma>0$, if the alternative (a) holds, $t$ is taken small enough, and $y_0,\ldots,y_n$ are arbitrary points
such that $y_k\in B(x_k,tr)$ for
$0\leq k \leq n$, denoting by $L_k^y$ the $k$-plane that passes through $y_0,y_1,\ldots,y_k$, we will have
$$\dist(y_k,L_{k-1}^y) \geq \frac12\,\gamma\,r\qquad\mbox{for $1\leq k\leq n$.}$$
For each $\gamma$, we denote by $t_0(\gamma)$ some constant $t$ such that this holds, and then we say that $B$ is $\gamma$-balanced if
it is $(t_0(\gamma),\gamma)$-balanced. Otherwise, we say that $B$ is $\gamma$-unbalanced.
\end{remark}


\vv

\begin{lemma}\label{lembalan2}
Let $\mu$ be a Radon measure in $\R^d$ and consider the dyadic lattice $\DD$ associated with
$\mu$ from Lemma \ref{lemcubs}. 
Given $0<t<1$ and $0<\gamma<1$ small enough (i.e.\ smaller than some absolute constant), there exists some constant $\ve=\ve(t)>0$ such that one of the following
alternatives holds for every $Q\in\DD^{db}$:
\begin{itemize}
\item[(a)] There are points $x_0,x_1,\ldots,x_n\in B(Q)=\frac1{28}B_Q$ such that 
$$\mu\bigl(B(x_k,t\,r(B(Q)))\cap B(Q)\bigr)\geq \ve\,\mu(Q) \qquad\mbox{for $0\leq k\leq n$,}$$
and if $L_k$ stands for the $k$-plane that passes through $x_0,x_1,\ldots,x_k$, then
$$\dist(x_k,L_{k-1}) \geq \gamma\,r(B(Q))\qquad\mbox{for $1\leq k\leq n$.}$$

\item[(b)] There exists a family of pairwise disjoint cubes $\{P\}_{P\in I_Q}\subset\DD^{db}$, which are contained in $Q$, so that $\ell(P) \gtrsim\gamma\,\ell(Q)$  and
$\Theta_\mu(2B_{P})\gtrsim \gamma^{-1}\,\Theta_\mu(2B_Q)$ for each $P\in I_Q$,
and
\begin{equation}\label{eqsgk32}
\sum_{P\in I_Q} \Theta_\mu(2B_{P})\,\mu(P)\gtrsim \gamma^{-1}\,\Theta_\mu(2B_Q)\,\mu(Q).
\end{equation}
\end{itemize}
\end{lemma}

Notice that in the previous lemma the cubes $Q$ and $P$, with $P\in I_Q$, are doubling.

\begin{proof}
Consider the measure $\sigma =\mu|_{B(Q)}$.
By applying Lemma \ref{lembalan} to $\sigma$ and the ball $B(Q)$ we infer that either
\begin{itemize}
\item[(i)]
there are points $x_0,x_1,\ldots,x_n\in B(Q)$ such that 
$$\sigma\bigl(B(x_k,t\,r(B(Q)))\bigr)\geq \ve(t)\,\sigma(B(Q)) \qquad\mbox{for $0\leq k\leq n$,}$$
and 
$$\dist(x_k,L_{k-1}) \geq \gamma\,\ell(Q)\qquad\mbox{for $1\leq k\leq n$,}$$
 
\item[(ii)] or there exists a family of balls $\{B_i\}_{i\in J_Q}$, 
 with radii $r(B_i)=4\gamma \,r(B(Q))$, centered on $B$, so that
the balls $\{10B_i\}_{i\in J_Q}$ are pairwise disjoint, 
$$\sum_{i\in J_Q}\sigma(B_i)\gtrsim \sigma(B(Q)),$$
and 
$$\Theta_\sigma(B_i)\gtrsim \gamma^{-1}\,\Theta_\sigma(B(Q))\quad \mbox{ for all $i\in J_Q$}.$$
\end{itemize}
If (i) holds, then the alternative (a) in the lemma holds, by adjusting suitably the constant~$\ve$.

Suppose now that the option (ii) holds. 
For each $i\in J_Q$ consider the cube $\wt P_i$ with $A_0^{-1}\gamma\ell(Q)<\ell(\wt P_i)\leq \gamma\,\ell(Q)$ which intersects $B_i$
and has maximal $\mu$-measure. 
From the fact that the balls $10B_i$, $i\in J_Q$, are pairwise disjoint
we deduce that the cubes $\wt P_i$, $i\in  J_Q$ are pairwise different. 
On the other hand, since $B_i\cap E$ can be covered by a bounded number of cubes with side length comparable
to $\gamma\,\ell(Q)$, we infer that 
$$\mu(\wt P_i)\geq \sigma(\wt P_i)\gtrsim \sigma(B_i)$$ and so 
\begin{equation}\label{eqdjk485}
\sum_{i\in J_Q}\mu(\wt P_i)\gtrsim \sigma(B(Q))\sim \mu(2B_Q),
\end{equation}
since $Q\in\DD^{db}$.
We also deduce that
\begin{equation}\label{eqdjkq18}
\Theta_\mu(2B_{\wt P_i})\gtrsim \Theta_\sigma(B_i)
\gtrsim  \gamma^{-1}\,\Theta_\sigma(B(Q))\sim \gamma^{-1}\,\Theta_\mu(2B_Q),
\end{equation}
taking into account again that $Q\in\DD^{db}$ in the last estimate.
Observe that the fact that $\mu(B(Q)\cap \wt P_i) =\sigma(\wt P_i)>0$ ensures that $\wt P_i\subset Q$.

For each $i\in J_Q$, consider the cube $P_i\in\DD^{db}$ which contains $\wt P_i$ and has minimal side length. Since $Q\in\DD^{db}$, such a cube exists and $P_i\subset Q$. From Lemma \ref{lemcad23} we infer that
\begin{equation}\label{eqdjk2894}
\Theta_\mu(2B_{P_i})\sim \Theta_\mu(100B(P_i))\gtrsim \Theta_\mu(100B(\wt P_i))
\gtrsim \Theta_\mu(2B_{\wt P_i}).
\end{equation}
We extract now from $\{P_i\}_{i\in \wt J_Q}$ the subfamily $I_Q$ of cubes which are maximal and thus disjoint. This family fulfills the properties stated in the alternative (b) of the lemma.
Indeed, by construction each $P\in I_Q$ satisfies
$\ell(P) \gtrsim\gamma\,\ell(Q)$  and since $P= P_i$ for some $i\in\wt J_Q$,
$$\Theta_\mu(2B_{P})\gtrsim \Theta_\mu(2B_{\wt P_i})\gtrsim
\gamma^{-1}\,\Theta_\mu(2B_Q),$$
recalling  \rf{eqdjk2894} and \rf{eqdjkq18}.
From the preceding estimate, \rf{eqdjk485}, and the fact that 
$\sum_{P\in I_Q} \mu(P)\geq \sum_{i\in J_Q}\mu(P_i)$, we infer that
$$
\gamma^{-1}\,\Theta_\mu(2B_Q)\,\mu(Q) \lesssim
\sum_{P\in I_Q} \Theta_\mu(2B_P)\,\mu(P).
$$
\end{proof}
\vv


\section{The Main Lemma}\label{secmlemma}
\vv


Let $\mu$ be the measure in Theorem \ref{teo1} and $E=\supp\mu$, and consider the
dyadic lattice associated with $\mu$ described in Section \ref{secdyad}.
Let $F\subset E$ be an arbitrary compact set such that
\begin{equation}\label{eqsk771}
\int_{F}\int_0^1 \beta_{\mu,2}(x,r)^2\,\frac{dr}r\, d\mu(x)<\infty.
\end{equation}


The next lemma concentrates the main difficulties for the proof of Theorem \ref{teo1}.
\vv

\begin{mainlemma}\label{mainlemma}
Let $0<\tau<1/100$ and $A>100$ be some fixed constants, with $\tau\ll A^{-1}$, say. Suppose that 
 $\delta$ and $\eta$ are small enough positive constants (depending only on $\tau$ and $A$).
Let $R\in\DD^{db}$ be a doubling cube with $\ell(R)\leq \delta$ such that 
\begin{equation}\label{eqassu1}
\mu(R\setminus F)\leq\eta\,\mu(R).
\end{equation}
Then there exists a bi-Lipschitz injection $g:\R^n\to \R^d$ with the  bi-Lipschitz constant 
bounded above by some absolute constant and a family of pairwise disjoint cubes $\sss(R)\subset \DD(R)$ such that the following holds. Consider the following subfamilies of $\sss(R)$:
\begin{itemize}
\item the high density family $\HD(R)$, which is made up of the cubes $Q\in\sss(R)$ which satisfy
$\Theta_\mu(2B_Q)\geq A\, \Theta_\mu(2B_R)$, 

\item the family $\LD(R)$ of low density  cubes, which is made up of the cubes $Q\in\sss(R)$ which satisfy
$\Theta_\mu(2B_Q)\leq \tau\, \Theta_\mu(2B_R)$,

\item the unbalanced family $\UB(R)$, which is made up of the cubes $Q\in\sss(R)\cap\DD^{db} \setminus (\HD(R)\cup\LD(R))
$ such that $\frac1{28}B_Q$ is $\tau^2$-unbalanced.
\end{itemize}
Let $\tree(R)$ the subfamily of the cubes from $\DD(R)$ which are not strictly contained in any cube
from $\sss(R)$. 
We have:
\vv
\begin{itemize}

\item[(a)] $\mu$-almost all $F\cap R\setminus \bigcup_{Q\in\sss(R)}Q$ is contained in 
$\Gamma_R:=g(\R^n)$ and, moreover, the restriction of  $\mu$ to $F\cap R\setminus \bigcup_{Q\in\sss(R)}Q$ is absolutely continuous with
respect to $\HH^1|_{\Gamma_R}$.
\vv

\item[(b)] For all $Q\in\tree(R)$, $\Theta_\mu(2B_Q)\leq c\,A\,\Theta_\mu(2B_R)$.\vv

\item[(c)] The following holds:
\begin{align*}
\sum_{Q\in\sss(R)\setminus (\HD(R)\cup\UB(R)) }\!\!\!\!\mu(Q)& \leq \tau^{1/2}\,\mu(R) \\
&\quad \!\!+ 
\frac{c(A,\tau)}{\Theta_\mu(2B_R)}\sum_{Q\in\tree(R)} \int_{F\cap\delta^{-1}B_Q}\int_{\delta\ell(Q)}^{\delta^{-1}\ell(Q)}
\beta_{\mu,2}(x,r)^2\,\frac{dr}r\,d\mu(x).
\end{align*}
\end{itemize}
\end{mainlemma}

\vvv
Let us remark that the assumption that $\ell(R)\leq\delta$ can be removed if we assume that
$$\int_{F}\int_0^\infty \beta_{\mu,2}(x,r)^2\,\frac{dr}r\, d\mu(x)<\infty,$$
instead of \rf{eqsk771}. Further, let us note that the lemma asserts that the family $\sss(R)$ contains
the subfamilies $\HD(R)$, $\LD(R)$ and $\UB(R)$ defined above. In fact, $\sss(R)$ will be constructed
in Section \ref{secfi} and will contain other subfamilies besides the preceding ones.

Note the difference with respect to Main Lemma 5.1 from \cite{Tolsa-delta}. Above we are  not able to estimate the measure of the 
high density cubes from $\HD(R)$. Instead it turns out that the cubes from $\LD(R)$ have very little mass (although this is 
not stated explicitly in the lemma). This is opposite to
what is shown in Lemma 5.1 from \cite{Tolsa-delta}, where the mass from the cubes from $\HD(R)$ is very small while
one cannot control the mass of the cubes from $\LD(R)$.


Before proving the Main Lemma \ref{mainlemma} we will show in the next section how Theorem \ref{teo1} follows from this, by
means of a suitable corona type decomposition.






\vv
\section{Proof of Theorem \ref{teo1} using the Main Lemma \ref{mainlemma}}\label{secprovam}
\vv

\subsection{Preliminaries}

To prove Theorem \ref{teo1} clearly it is enough to show that $\mu|_F$ is rectifiable. 
Further we may and will assume that
\begin{equation}\label{eqc*}
M_n\mu(x) = \sup_{r>0} \frac{\mu(B(x,r))}{r^n}\leq C_*\quad \mbox{ for all $x\in F$},
\end{equation}
for some constant $C_*$ big enough.
Let $x_0$ be a
point of density of $F$ and for $\eta>0$ let $B_0=B(x_0,r_0)$ be some ball such that
\begin{equation}\label{eqdd3}
\mu(B_0\setminus F)\leq\eta^2\mu(B_0)\qquad \mbox{and}\qquad \mu(\tfrac12B_0)
\geq \frac1{2^{d+1}}\,\mu(B_0).
\end{equation}
Taking into account that for $\mu$-almost every $x_0\in F$ there exists a sequence of balls
like $B_0$ (i.e.\ fulfilling \rf{eqdd3})
centered at $x_0$ with radius tending to $0$
 (see Lemma 2.8 of \cite{Tolsa-llibre} for example), it suffices to prove that any ball like $B_0$ contains a 
rectifiable subset with positive $\mu$-measure. 

Denote by 
$\BZ$ 
the family of cubes $R\in\DD$ contained in $B_0$ such that 
$$\mu(R\setminus F)\geq \eta\,\mu(R).$$
Next we show that the union of the cubes from $\BZ$ has very small $\mu$-measure.
\vv

\begin{lemma}\label{lemaux21}
We have
\begin{equation}\label{eqmaxx4}
\mu\Biggl(\bigcup_{R\in\BZ} R\Biggr)\leq c\,\eta\,\mu(B_0).
\end{equation}
\end{lemma}

\begin{proof}
We consider the maximal dyadic 
operator
\begin{equation}\label{eqmd*}
M^df(x) = \sup_{Q\in\DD:x\in Q}\frac1{\mu(Q)}\int|f|\,d\mu,
\end{equation}
which is bounded from $L^1(\mu)$ to $L^{1,\infty}(\mu)$.
From \rf{eqassu1} we get
$$\mu\Biggl(\bigcup_{R\in\BZ} R\Biggr)\leq \mu\bigl(\{x\in\R^d:\,M^d\chi_{B_0\setminus F}
(x)\geq \eta\}\bigr)\leq c\,\frac{\mu(B_0\setminus F)}\eta \leq c\,\eta\,\mu(B_0),$$
as wished.
\end{proof}

\vv


\subsection{The families $\wt {\HD}(R)$, $\wt {\UB}(R)$, $\wt {\mathsf{O}}(R)$ and $\wt\sss(R)$}

Recall that the Main Lemma asserts that if $R\in\DD^{db}$, with $\ell(R)\leq\delta$, satisfies the assumption \rf{eqassu1}, then it generates some families of cubes $\HD(R)$, $\UB(R)$ and
$\sss(R)$ fulfilling the properties (a), (b) and (c). 
In this subsection we will introduce some variants of these families. First we need the following auxiliary result.

\begin{lemma}\label{lemhd*}
Assuming $A$ big enough, for every $Q\in\HD(R)$ there exists $P(Q)\in\DD^{db}$ which contains $Q$ with $\ell(P(Q))\sim \ell(Q)$ and
$\Theta_\mu(P(Q))\sim\Theta_\mu(Q)$.
\end{lemma}

\begin{proof}
Let $P(Q)\in\DD$ be the smallest ancestor of $P$ which belongs to $\DD^{db}$. Such a cube $P(Q)$ exists and
$P(Q)\subset R$ because $R\in\DD^{db}$.
 For $i\geq 0$, denote by $Q_i$ be the $i$-th ancestor of $Q$ (i.e.\ $Q_i\in\DD$ is such that $Q\subset Q_i$ and $\ell(Q_i)=A_0^i\,\ell(Q)$).
Let $i\geq 0$ be such that $P(Q)= Q_i$.
Since the cubes $Q_1,\ldots,Q_{i-1}$ do not belong to $\DD^{db}$, by Lemma \ref{lemcad23} we have
$$A\,\Theta_\mu(2B_R)\leq\Theta_\mu(2B_Q)\lesssim\Theta_\mu(100B(Q))\leq C_0\,A_0^{-9ni}\,\Theta_\mu(100B(Q_i))
\sim A_0^{-9ni}\,\Theta_\mu(2B_{Q_i}).$$
As $Q_i\in\tree(R)$, we have $\Theta_\mu(2B_{Q_i})\leq c\,A\,\Theta_\mu(2B_R)$, and the estimate above
implies that $i\lesssim1$. That is, $\ell(P(Q))\sim1$, which in turn gives that
$\Theta_\mu(2B_{P(Q)})\gtrsim \Theta_\mu(2B_Q)$ and proves the lemma.
\end{proof}
\vvv

We define the family $\HD^*(R)$ as follows:
$$\HD^*(R) = \bigl\{P(Q):\,Q\in\HD(R)\bigr\},$$
where $P(Q)$ is as in Lemma \ref{lemhd*}.

Now we turn our attention to the family $\UB(R)$.
Recall that, by Lemma \ref{lembalan2}, if $Q\in\UB(R)$,
there exists a family of pairwise disjoint cubes $\{P\}_{P\in I_Q}\subset\DD^{db}$, which are contained in $Q$, so that $\ell(P) \gtrsim  \tau^2\,\ell(Q)$  and
$\Theta_\mu(2B_{P})\gtrsim \tau^{-2}\,\Theta_\mu(2B_Q)$ for each $P\in I_Q$,
and
\begin{equation}\label{eqsgk32**}
\sum_{P\in I_Q} \Theta_\mu(2B_{P})\,\mu(P)\gtrsim \tau^{-2}\,\Theta_\mu(2B_Q)\,\mu(Q).
\end{equation}
We consider a family $\wt I_Q$ of cubes contained in $Q$, with side length comparable to $a\,\ell(Q)$,  
disjoint from the ones from $I_Q$, so that 
$$Q=\bigcup_{P\in I_Q\cup\wt I_Q} P.$$
We define
$$\UB^*(R) = \bigcup_{Q\in\UB(R)} (I_Q \cup \wt I_Q).$$

On the other hand, we denote
$${\mathsf{O}}(R) = \sss(R)\setminus (\HD(R)\cup\UB(R)).$$
We set
$${\mathsf{O}}^*(R) = \{Q\in\DD:\,\mbox{$Q$ is the son of some cube from ${\mathsf{O}}(R)$}\},$$
and 
$$\sss^*(R) =\HD^*(R) \cup\UB^*(R) \cup {\mathsf{O}}^*(R).$$
Finally, let $\wt\sss(R)$ be a maximal subfamily (and thus disjoint) of $\sss^*(R)$.
We denote by $\wt{\HD}(R)$, $\wt{\UB}(R)$ and $\wt {\mathsf{O}}(R)$ the subfamilies of the cubes from $\wt\sss(R)$ which belong to 
$\HD^*(R)$, $\UB^*(R)$ and ${\mathsf{O}}^*(R)$, respectively.

\begin{remark}\label{rem111}
Notice that, by construction, if $Q \in\UB(R)$ is not contained in any cube from $\wt {\HD}(R)$, then all the cubes from $I_Q\cup\wt I_Q$ belong to $\wt\sss(R)$. Observe also that $R$ does not belong to $\wt\sss(R)$. Indeed, $R\not\in \wt {\HD}(R)$ because every $Q\in\HD(R)$ satisfies
$\Theta_\mu(2B_Q)\gtrsim A\,\Theta_\mu(2B_R)$, so for $A$ large enough, $\Theta_{\mu}(2B_{Q})>\Theta_{"\mu}(2B_{R})$. On can also deduce that $R\not\in\wt{\UB}(R)$ from \rf{eqsgk32**} for $a$ small enough. On the other hand,  $R\not\in\wt {\mathsf{O}}(R)$ because all the cubes from $\wt {\mathsf{O}}(R)$ are sons of other cubes from $\tree(R)$.
\end{remark}

\vv


\subsection{The corona decomposition}

Let us continue with the proof of Theorem \ref{teo1}.
From \rf{eqmaxx4} and the fact that $\mu(B_0)\sim\mu(\frac12\,B_0)$ we infer that,
for $\eta$ small enough, there exists some cube $R_0\in\DD^{db}$ satisfying $R_0\subset \frac34 B_0$,
$\ell(R_0)\leq\delta$, $\delta^{-1}B_{R_0}\subset \frac9{10}B_0$, and 
$$\mu\Biggl(R_0\setminus \bigcup_{Q\in\BZ}Q\Biggr)>0.$$
We are going now to construct a family of cubes $\ttt$ contained in $R_0$ inductively, by applying the Main
Lemma \ref{mainlemma}. To this end, we need to introduce some additional notation.

Above, for a cube $R\in\DD^{db}$, with $\ell(R)\leq\delta$, which satisfies \rf{eqassu1}, we have defined a family of stopping cubes $\wt\sss(R)$.
Now it is convenient to define $\wt\sss(R)$
also if the assumption \rf{eqassu1} does not hold.
If $R$ is a descendant of $R_0$ such that $R\in\DD^{db}\cap\BZ
$ (note that
this means that \rf{eqassu1} does not hold), we set $\wt\sss(R)=\varnothing$. 

Given a cube $Q\in\DD$, we denote by $\MD(Q)$ the family  of maximal cubes 
(with respect to inclusion) from $P\in \DD^{db}(Q)$. Recall that, by Lemma \ref{lemcobdob}, this family covers $\mu$-almost
all $Q$. Moreover, by Lemma \ref{lemcad23} it follows that if $P\in\MD(Q)$, then $\Theta_\mu(2B_P)
\leq c\,\Theta_\mu(2B_Q)$.
Given $R\in\ttt$, we denote
$$\nex(R) = \bigcup_{Q\in\wt\sss(R)} \MD(Q).$$
By Remark \ref{rem111} and the construction above, it is clear that the cubes from $\nex(R)$ are different from $R$.

For the record, notice that, by construction, if $P\in\nex(R)$, then
\begin{equation}\label{eqdens***}
\Theta_\mu(2B_S)\leq c(A,\tau)\,\Theta_\mu(2B_R) \quad\mbox{for all $S\in\DD$ such that $P\subset S\subset R$.}
\end{equation}

We are now ready to construct the aforementioned family $\ttt$. We will have
$\ttt=\bigcup_{k\geq0}\ttt_k$. First we set
$$\ttt_0=\{R_0\}.$$
Assuming $\ttt_k$ to be defined, we set
$$\ttt_{k+1}  = \bigcup_{R\in\ttt_k} \nex(R).$$
Note that the families $\nex(R)$, with $R\in\ttt_k$, are pairwise disjoint.


\subsection{The families of cubes $ID_H$, $ID_U$ and $ID$}

We distinguish a special type of cubes $R\in\ttt$. We write $R\in ID_H$ (increasing density because of high density cubes) if
$$\mu\biggl(\,\bigcup_{Q\in\wt{\HD}(R)} Q\biggr)\geq \frac14  \,\mu(R).
$$
Also, we write $R\in ID_U$ (increasing density because of unbalanced cubes) if
$$\mu\biggl(\,\bigcup_{Q\in\wt{\UB}(R)} Q\biggr)\geq \frac14  \,\mu(R).
$$
We set
$$ID = ID_H\cup ID_U.$$
\vv

\begin{lemma}\label{lemID}
Suppose that $A$ is big enough and $\tau$ small enough. If $R\in ID$, then
\begin{equation}\label{eqaki33}
\Theta_\mu(2B_R)\,\mu(R) \leq \frac12 \sum_{Q\in \nex(R)}\Theta_\mu(2B_Q)\,\mu(Q).
\end{equation}
\end{lemma}

\begin{proof}
Suppose first that $R\in ID_H$. Recalling that $\Theta_\mu(2B_Q)\gtrsim A\,\Theta_\mu(2B_R)$ for every $Q\in\wt{\HD}(R)$,
we deduce that
 $$\Theta_\mu(2B_R)\,\mu(R) \leq 4 \sum_{Q\in \wt{\HD}(R)}\Theta_\mu(2B_R)\,\mu(Q)
\leq c\,A^{-1}\sum_{Q\in \wt{\HD}(R)}\Theta_\mu(2B_Q)\,\mu(Q).$$
Since the cubes from $\wt{\HD}(R)$ belong to $\DD^{db}$ it follows that $\wt{\HD}(R)\subset\nex(R)$ and then
 it is clear that \rf{eqaki33} holds if $A$ is taken big enough.

Consider now the case that $R\in ID_U$, and take $Q\in\wt{\UB}(R)$. By construction, there exists a cube $\wh Q\in
\UB(R)$ which contains $Q$, with $Q\in I_{\wh Q}\cup\wt I_{\wh Q}$, and such that all the cubes from $I_{\wh Q}\cup\wt I_{\wh Q}$ belong to $\wt\sss(R)$. Since the cubes from $I_{\wh Q}$ are doubling,
it turns out that $I_{\wh Q}\subset \nex(R)$. 
Denote by $\UB_0(R)$ the cubes from $\UB(R)$ which contain some cubes from $\wt {\UB}(R)$ (they coincide with the cubes
from $\UB^{*}(R)$ which are not contained in any cube from $\wt {\HD}(R)$).
Then we have
$$
\Theta_\mu(2B_R)\,\mu(R) \leq 4 \sum_{Q\in \wt{\UB}(R)}\Theta_\mu(2B_R)\,\mu(Q) = 
4 \sum_{S\in\UB_0(R)}\Theta_\mu(2B_R)\,\mu(S).
$$
Using now that $\Theta_\mu(2B_R)\leq  \tau^{-1}\,\Theta_\mu(2B_S)$ (since the cubes from $\UB(R)$ do not belong to $\LD(R)$)
and recalling \rf{eqsgk32**}, we infer that
$$\Theta_\mu(2B_R)\,\mu(R)\leq 4\, \tau^{-1}\sum_{S\in\UB_0(R)}\Theta_\mu(2B_S)\,\mu(S)
\lesssim \tau^{-1} \tau^2 \sum_{S\in\UB_0(R)}
\sum_{Q\in I_S} \Theta_\mu(2B_{Q})\,\mu(Q).
$$
Since for all  $S\in\UB_0(R)$ the cubes from $I_S$ belong to $\DD^{db}$, we have $I_S\subset \nex(R)$ , and thus \rf{eqaki33}
also holds in this case if $\tau$ is small enough. 
\end{proof}
\vv


\subsection{The packing condition}

Next we prove a key estimate.
\vv

\begin{lemma}\label{lemkey62}
If $\tau$ is chosen small enough in the Main Lemma, then
\begin{equation}\label{eqsum441}
\sum_{R\in\ttt}\Theta_\mu(2B_R)\,\mu(R)\leq C_*\,\mu(R_0)+
c(A,\tau,\eta,\delta)\,\int_F\int_0^1\beta_{\mu,2}(x,r)^2\,\frac{dr}r\,d\mu(x),
\end{equation}
where $C_*$ is the constant in \rf{eqc*}.
\end{lemma}

\begin{proof}
For a given $k\leq 0$, we denote 
$$\ttt_0^k = \bigcup_{0\leq j\leq k}\ttt_j,$$
and also
$$ID_0^k = ID \cap \ttt_0^k.$$

To prove \rf{eqsum441}, first we deal with the cubes from the $ID$ family.
By Lemma \ref{lemID}, for every $R\in ID$ we have
$$\Theta_\mu(2B_R)\,\mu(R) \leq \frac12 \sum_{Q\in\nex(R)} \Theta_\mu(2B_Q)\,\mu(Q)$$
 and hence we obtain
$$
\sum_{R\in ID_0^k} \Theta_\mu(2B_R)\,\mu(R) \leq \frac12 \sum_{R\in ID_0^k}\sum_{Q\in \nex(R)}\Theta_\mu(2B_Q)\,\mu(Q)
\leq \frac12\sum_{Q\in \ttt_0^{k+1}}\Theta_\mu(2B_Q)\,\mu(Q),
$$
because the cubes from $\nex(R)$ with $R\in\ttt_0^k$ belong to $\ttt_0^{k+1}$. 
So we have
\begin{align*}
\sum_{R\in \ttt_0^k} \Theta_\mu(2B_R)\,\mu(R) & = \sum_{R\in \ttt_0^k\setminus ID_0^k} \Theta_\mu(2B_R)\,\mu(R)+
\sum_{R\in ID_0^k} \Theta_\mu(2B_R)\,\mu(R)\\
& \leq \!\sum_{R\in \ttt_0^{k}\setminus ID_0^k}\!\!\!\Theta_\mu(2B_R)\,\mu(R) +
\frac12\sum_{R\in \ttt_0^{k}}\!\Theta_\mu(2B_R)\,\mu(R) +
 c\,C_*\mu(R_0),
\end{align*}
where we took into account that $\Theta_\mu(2B_R)\lesssim C_*$ for every $R\in\ttt$ (and in particular for all $R\in\ttt_{k+1}$)
for the last inequality.
So we deduce that
$$\sum_{R\in \ttt_0^k} \Theta_\mu(2B_R)\,\mu(R)  \leq 2\sum_{R\in \ttt_0^{k}\setminus ID_0^k}\Theta_\mu(2B_R)\,\mu(R) +
 c \,C_*\mu(R_0).$$
Letting $k\to\infty$, we derive
\begin{equation}\label{eqaaii}
\sum_{R\in \ttt} \Theta_\mu(2B_R)\,\mu(R)\leq 2\sum_{R\in \ttt\setminus ID} \Theta_\mu(2B_R)\,\mu(R)+  c \,C_*\mu(R_0).
\end{equation}

We split the first term on the right hand side of \rf{eqaaii} as follows:
\begin{align}\label{eqsssi}
\sum_{R\in \ttt\setminus ID} \Theta_\mu(2B_R)\,\mu(R) & = \sum_{R\in \ttt\setminus (ID\cup \BZ)}\!\!\!\cdots \;
+ \sum_{R\in \ttt\cap \BZ} \!\!\!\cdots =: S_1 + S_2.
\end{align}
To estimate the sum $S_1$ we use the fact that, for $R\in\ttt\setminus (ID\cup \BZ)$, 
we have
$$
\mu\biggl(R\setminus \bigcup_{Q\in \wt{\HD}(R)\cup \wt{\UB}(R)} Q\biggr)\geq \frac12\,\mu(R),$$
and then we apply the inequality (c) in the Main Lemma to get
\begin{align*}
\mu(R)& \leq 
2\,\mu\biggl(R\setminus \bigcup_{Q\in \wt\sss(R)} Q\biggr) 
+ 2
\,\mu\biggl(
\bigcup_{Q\in\wt\sss(R)\setminus \wt{\HD}\cup\wt{\UB}(R) }Q\biggr)\\
& \leq 
2\,\mu\biggl(R\setminus \bigcup_{Q\in \nex(R)} Q\biggr) 
+ 2
\sum_{Q\in\sss(R)\setminus (\HD(R)\cup\UB(R)) }\mu(Q)
\\
& \leq 2 \,\mu\biggl(R\setminus \bigcup_{Q\in \nex(R)} Q\biggr) 
+2\tau^{1/2}\,\mu(R) \\&
\quad + 
\frac{c(A,\tau)}{\Theta_\mu(2B_R)}\sum_{Q\in\tree(R)} \int_{F\cap\delta^{-1}B_Q}\int_{\delta\ell(Q)}^{\delta^{-1}\ell(Q)}
\beta_{\mu,2}(x,r)^2\,\frac{dr}r\,d\mu(x).
\end{align*}

Hence, if $\tau^{1/2}\leq 1/4$, say,
$$\mu(R) \leq 4 \,\mu\biggl(R\setminus \bigcup_{Q\in \nex(R)} Q\biggr) 
+
\frac{c(A,\tau)}{\Theta_\mu(2B_R)}\sum_{Q\in\tree(R)} \int_{F\cap\delta^{-1}B_Q}\int_{\delta\ell(Q)}^{\delta^{-1}\ell(Q)}
\beta_{\mu,2}(x,r)^2\,\frac{dr}r\,d\mu(x).$$
So we deduce that
\begin{align}\label{eqal163}
S_1 & \leq 4 
\sum_{Q\in\tree(R)} \Theta_\mu(2B_R)\,\mu\biggl(R\setminus \bigcup_{Q\in \nex(R)} Q\biggr)
\nonumber\\
&\quad
+
c(A,\tau) \sum_{R\in \ttt}\,
\sum_{Q\in\tree(R)} \int_{F\cap\delta^{-1}B_Q}\int_{\delta\ell(Q)}^{\delta^{-1}\ell(Q)}
\beta_{\mu,2}(x,r)^2\,\frac{dr}r\,d\mu(x).
\end{align}
To deal with the first sum on the right hand side above we take into account that $\Theta_\mu(2B_R)\lesssim C_*$ for all $R\in\ttt$ by \eqref{eqc*} and that the sets $R\setminus \bigcup_{Q\in \nex(R)} Q$, with $R\in\ttt$, are pairwise disjoint. Then we get
$$\sum_{Q\in\tree(R)} \Theta_\mu(2B_R)\,\mu\biggl(R\setminus \bigcup_{Q\in \nex(R)} Q\biggr)\leq c\,C_*\,\mu(R_0).$$
On the other hand, the last sum on the right hand side of \rf{eqal163} does not exceed
$$
 \sum_{Q\in \DD}
 \int_{F\cap\delta^{-1}B_Q}\int_{\delta\ell(Q)}^{\delta^{-1}\ell(Q)}
\beta_{\mu,2}(x,r)^2\,\frac{dr}r\,d\mu(x)
\leq c(\delta)\,\int_F\int_0^1\beta_{\mu,2}(x,r)^2\,\frac{dr}r\,d\mu(x),
$$
by the finite superposition of the domains of integration of the integrals on the left hand side.
So we obtain 
$$S_1\leq c\,C_*\,\mu(R_0) +
 c(A,\tau,\delta) \int_F\int_0^1\beta_{\mu,2}(x,r)^2\,\frac{dr}r\,d\mu(x).
$$

Concerning the sum $S_2$ in \rf{eqsssi} we take into account that, by construction, the cubes $R\in\ttt\cap \BZ$
are pairwise disjoint, because $\nex(R) = \varnothing$ for such cubes $R$. So we have
$$S_2 \leq c\,C_*\sum_{R\in\ttt\cap \BZ}\mu(R) \leq c\,C_*\,\mu(R_0),$$
as $\Theta_\mu(2B_R)\lesssim C_*$ for every $R\cap\ttt$.

Gathering the estimates we obtained for $S_1$ and $S_2$ and applying \rf{eqaaii}, the lemma follows.
\end{proof}

\vv


\subsection{Proof of Theorem \ref{teo1}}\label{subhh*}
From Lemma \ref{lemkey62} we deduce that for $\mu$-a.e.\ $x\in R_0$,
\begin{equation}\label{eqsug9}
\sum_{R\in\ttt:x\in R} \Theta_\mu(2B_R)<\infty.
\end{equation}
For a given $x\in R_0\setminus \bigcup_{Q\in\BZ}Q$ such that
\rf{eqsug9} holds, 
let $R_0,R_1,R_2,\ldots$ be the cubes from $\ttt$ such that $x\in R_i$. Suppose that
this is an infinite sequence and assume that $R_0\supset R_1\supset R_2\supset\ldots$,
so that for each $i\geq0$, $R_{i+1}\in\nex(R_i)$.
From \rf{eqdens***} it follows that
$$\Theta_\mu(x,r)\l\Theta_\mu(x,r)eq c(A,\tau)\,\Theta_\mu(2B_{R_i})\qquad \mbox{ for
$\frac1{10}\ell(R_{i+1})\leq r\leq \frac1{10}\ell(R_i)$.}$$
As a consequence, 
$$\Theta^{n,*}(x,\mu)\leq c(A,\tau)\,\limsup_{i\to\infty}\Theta_\mu(2B_{R_i}).$$
From \rf{eqsug9}, we infer that the limit on the right hand side above
vanishes and thus $\Theta^{n,*}(x,\mu)=0$.
So we have shown that for any $x\in R_0$ satisfying \rf{eqsug9}, the condition $\Theta^{n,*}(x,\mu)>0$ implies that the collection
of cubes $R\in\ttt$ which contain $x$ is finite.

By the property (a) in the Main Lemma and the above construction, if $R\in\ttt\setminus
\BZ$, then there exists a set $Z_R$ of $\mu$-measure $0$ and a set 
$W_R\subset \Gamma_R$ such that
\begin{equation}\label{eqfhe2}
R\subset Z_R\cup W_R\cup\bigcup_{Q\in\ttt(R)} Q,
\end{equation}
with $\mu|_{W_R}$ being absolutely continuous with respect to $\HH^1|_{\Gamma_R}$.

Suppose now that $\Theta^{n,*}(x,\mu)>0$, that
\begin{equation}\label{eqshh1}
x\in R_0\setminus \Biggl(\bigcup_{R\in \ttt} Z_R\cup \bigcup_{Q\in\BZ}Q\Biggr),
\end{equation} 
and that \rf{eqsug9} holds. Note that the set of such points is a subset of full $\mu$-measure of 
$R_0\setminus \bigcup_{Q\in\BZ}Q$. Let $R_n$ be the smallest cube from $\ttt$ which
contains $x$. Since $x\not\in\bigcup_{Q\in\BZ}Q$, we have $R_n\not\in \BZ$ and 
so \rf{eqfhe2} holds for $R_n$. Since $x\not\in Z_{R_n}$ and 
$x$ does not belong to any cube from $\nex(R_n)$ (by the choice of $R_n$), we infer that 
 $x\in W_{R_n}\subset
\Gamma_{R_n}$.
Thus $\mu$-almost all the subset of points $x$ with $\Theta^{n,*}(x,\mu)>0$
satisfying \rf{eqshh1} and \rf{eqsug9} is contained in $\bigcup_{n}W_{R_n}$, which is an $n$-rectifiable set such
that $\mu|_{\bigcup_{n}W_{R_n}}$ is absolutely continuous with respect to $\HH^n|_{\bigcup_{n}W_{R_n}}$.
\fiproof
\vv


\section {Proof of the Main Lemma} \label{secfi}

\vv

\def\dist{\mathrm{dist}}
\def\spn{\mathrm{Span}}

\def\lec{\lesssim}
\def\LF{\mathsf{LF}}

In this section, we prove the Main Lemma. We will assume that all implicit constants in the inequalities that follow depend on $C_{0},A_{0}$, and $d$.

\subsection{The stopping conditions}

Take $R\in\DD^{db}$ with $\diam (R)\leq \ell(R)\leq\delta<1/2$ so that 
\rf{eqassu1} holds. We denote by $\sss(R)$ the family of the maximal cubes $Q\subset R$  for which one of the following holds:
\begin{enumerate}
\item $Q\in \HD(R)\cup \LD(R)\cup \UB(R)$,
\item $Q\in \LF(R)$ (``low concentration of $F$'') where $\LF(R)$ is the set of cubes $Q\subset R$ for which
\[\mu(Q\cap F)\leq \mu(Q)/2,\]
\item $Q\in \mathsf{J}(R)$ (``big Jones' function'') , meaning $Q\not\in (\HD(R)\cup \LD(R)\cup \UB(R)\cup \LF(R))$ and
\[\sum_{Q\subset Q'\subset R} \beta(Q')^2\geq \alpha^{2}\]
where $\alpha>0$ is a number we will pick later and
\[\beta(Q')^{2}\Theta_\mu(2B_{R}):=\int_{\delta^{-1}B_{Q'}\cap F}\int_{\delta\ell(Q')}^{\delta^{-1}\ell(Q')}\frac{\beta_{\mu,2}(x,r)^{2}}{\mu(Q'\cap F)}\frac{dr}{r}d\mu(x) .\]
\end{enumerate}
For the reader's convenience, let us say that we will choose $\alpha\ll\min(\tau,A^{-1})$.

Recall that $\tree(R)$ is the subfamily of the cubes from $\DD(R)$ which are not strictly contained in any cube
from $\sss(R)$.
The following statement is an immediate consequence of the construction of $\sss(R)$ and $\tree(R)$.

\begin{lemma}\label{lemdobbb1}
If $Q\in\DD$, $\ell(Q)\leq\ell(R)$, and $Q\in \tree(R)\backslash \sss(R)$, then $Q\not\in\LF(R)$ and
$$\tau\,\Theta_\mu(2B_R)\leq \Theta_\mu(2B_Q) \leq A\,\Theta_\mu(2B_R).$$
If moreover $Q\in\DD^{db}$, then $\frac1{28}B_Q$ is $\tau^2$-balanced.
\end{lemma}

From now on, we say that $Q\in\DD$ is balanced if $B(Q)=\frac1{28}B_Q$ is $\tau^2$-balanced, or just balanced. Otherwise, we say that it is $\tau^2$-unbalanced, or just unbalanced

Note that for $\eta$ small enough, we can guarantee that $R\not\in \sss(R)$. Moreover, observe that the cubes in $\LF(R)$ are disjoint 
and so
\begin{equation}\label{sumLF}
\sum_{Q\in \LF(R)}\mu(Q)
\leq 2\sum_{Q\in \LF(R)} \mu(Q\backslash F)
\leq 2\mu(R\backslash F)<2\alpha.
\end{equation}
Just as well, the cubes in $\mathsf{J}(R)$ are disjoint and thus

\begin{align*}
\Theta_\mu(2B_{R})\alpha^{2}\sum_{Q\in J(R)}\mu(Q)
& \leq  \sum_{Q\in J(R)}\sum_{Q\subset Q'\subset R} \int_{\delta^{-1}B_{Q'}\cap F}\int_{\delta\ell(Q')}^{\delta^{-1}\ell(Q')} \beta_{\mu,2}(x,r)^{2}\frac{dr}{r}d\mu(x) \frac{\mu(Q)}{\mu(Q'\cap F)}\\
& \leq \sum_{Q'\in \tree(R)\backslash \LF(R)} \int_{\delta^{-1}B_{Q'}\cap F}\int_{\delta\ell(Q')}^{\delta^{-1}\ell(Q')} \beta_{\mu,2}(x,r)^{2}\frac{dr}{r}d\mu(x)\frac{\mu(Q')}{\mu(Q' \cap F)}\\
& \leq  2 \sum_{Q'\in \tree(R)} \int_{\delta^{-1}B_{Q'}\cap F}\int_{\delta\ell(Q')}^{\delta^{-1}\ell(Q')}\beta_{\mu,2}(x,r)^{2}\frac{dr}{r}d\mu(x).
\end{align*}

\vv
\subsection{The theorem of David and Toro}

All that remains to show is that we can cover the portion of $F\cap R$ not in any stopped cube by a bi-Lipschitz image 
of $\R^n$ and to control the sum of the cubes in $\LD(R)$, that is,
\begin{equation}\label{sumld}
\sum_{Q\in \LD(R)}\mu(Q)<\tau^{\frac{1}{2}}\mu(R).
\end{equation}

\vv
The main ingredient to proving these two facts is a theorem of David and Toro.  To state this, we need some 
additional notation. Given  two closed sets $E$ and $F$, $x\in \mathbb{R}^{d}$, and $r>0$, we denote
\[
d_{x,r}(E,F)=\frac{1}{r}\max\left\{\sup_{y\in E\cap B(x,r)}\dist(y,F), \sup_{y\in F\cap B(x,r)}\dist(y,E)\right\}.\]
\vv

\begin{theorem}[David, Toro]
For $k\in \N\cup\{0\}$, set $r_{k}=10^{-k}$ and let $\{x_{jk}\}_{j\in J_{k}}$ be a collection of points so that for some $n$-plane $V$,
\[\{x_{j0}\}_{j\in J_{0}}\subset V,\] 
 \[|x_{ik}-x_{jk}|\geq r_{k},\]
and, denoting $B_{jk}=B(x_{j,k},r_{k})$, 
\begin{equation}
x_{ik}\in \bigcup_{j\in J_{k-1}}2B_{jk-1}.
\label{V2}
\end{equation}
To each point $x_{jk}$ associate an $n$-plane $L_{jk}\subset \R^d$  and set
$$
\ve_{k}(x)=\sup\{d_{x,10^{4}r_{l}}(L_{jk},L_{il}): j\in J_{k}, |l-k|\leq 2, i\in J_{l}, x\in 100 B_{jk}\cap 100 B_{il}\}.
$$
There is $\ve_{0}>0$ such that if $\ve\in (0,\ve_{0})$ and
\begin{equation}
\sum_{k\geq 0} \ve_{k}(x)^{2}<\ve\mbox{ for all }x\in \R^{n},
\label{ek}
\end{equation}
then there is an $L$-bi-Lipschitz injection $g:\R^{n}\rightarrow \R^{d}$, with $L=L(n,d)$, so that the set 
\begin{equation}\label{eqahj0}
E_{\infty}=\bigcap_{K=1}^{\infty}\overline{\bigcup_{k= K}^{\infty} \{x_{jk}\}_{j\in J_{k}}}
\end{equation}
is contained in $g(\R^{n})$.

Moreover, $g(x)=\lim_k f_{k}(x)$ where $|f_{k}(x_{jk})-x_{jk}|\lec \ve r_{k}$, and 
\begin{equation}\label{closetog}
\dist(x,g(\R^{d}))\lesssim \ve r_{k}\mbox{\; for all\; }x\in 40B_{jk}\cap L_{jk}.\end{equation}
\label{DT}
\end{theorem}

This theorem is a slight restatement of Theorem 2.5 in \cite{DT1}, where the last inequality follows from Proposition 5.1 and equation (6.8) in the same paper.
In \cite{DT1}, the points $x_{j0}$ are allowed to be near some surface $\Sigma_{0}$ (see (2.7) in that paper), so in our case $V=\Sigma_{0}$. Moreover, in our application below, $R$ is assumed to have diameter less than $1$, and so $\{x_{j0}\}_{k\in J_{0}}$ consists of a single point, and thus the condition that $\{x_{j0}\}_{k\in J_{0}}\subset V$ for some $n$-plane $V$ is trivially satisfied. 

\def\bta{\beta_{\mu,2}}

We would like to point out here the versatility of this result. While there is some technical effort to adapting this theorem to our scenario, it is very natural for stopping-time arguments. Traditionally, given a Reifenberg flat topological surface for example, the points $\{x_{jk}\}_{j\in J_{k}}$ are taken to be a nested sequence of maximal $r_{k}$-nets in the surface. The way the theorem is stated, however, does not require this. In fact, the theorem only requires that, when we pick maximally separated points at each scale $r_{k}$, that they are close to points chosen for the previous scale (see \eqref{V2}). Thus, in choosing these $x_{jk}$, we can stop adding points in a specific region and add points elsewhere, much like a stopping-time process.

Our goal now is to pick appropriate choices of $x_{jk}$ and $L_{jk}$ for Theorem \ref{DT}. Roughly speaking, the points $x_{jk}$ correspond to centers of the cubes $Q\in \tree(R)$ and the $n$-planes $L_{jk}$ to the best approximating $n$-planes, and our control on the $\beta_{\mu,2}$-numbers in $\tree(R)$ will help us control $\ve_{k}$. However, this is not quite true, since, for example, the best approximating plane might not pass through or even close to the center of the cube, our cubes decrease at a much faster rate than just $r_{k}$, and moreover, not every cube $Q\in \tree(R)$ is balanced, which is a crucial property we will need to control the angles between nearby planes. Thus, there are many adjustments to be made.

We first remedy the issue of not all cubes being balanced by showing that, for any cube, there is always an ancestor close by that is balanced.

\vv
\begin{lemma}\label{lemdob32}
There is some constant $c_2(A,\tau)>0$ small enough so that for any cube 
$Q\in\tree(R)$  there exists some cube $Q'\supset Q$ such that $Q'\in\DD^{db}\cap \tree(R)\setminus\sss(R)$ (so $Q$ is balanced)
and $\ell(Q')\leq c_2(A,\tau)\,\ell(Q)$.
\end{lemma}

\begin{proof}
Let $Q=Q_0\subset Q_1\subset Q_2\ldots$ be cubes such that each $Q_i$ is son of $Q_{i+1}$.
If $Q_0,Q_1,\ldots ,Q_i$ are not doubling, from the fact that $R\in\DD^{db}$ it follows that al the cubes 
$Q_1,\ldots ,Q_i$ belong to $\tree(R)\setminus\sss(R)$. By Lemma \ref{lemcad23} we have
\begin{align*}
\tau \Theta_\mu(2B_R)& \leq \Theta_\mu(2B_{Q_1})\lesssim
\Theta_\mu(100B(Q_1))\\ &\lesssim A_0^{-9n\, (j-1)}\,\Theta_\mu(100B(Q_j))\lesssim c\,A_0^{-9n\, (j-1)}\,
\Theta_\mu(2B_{Q_{i+1}}),
\end{align*}
and thus $\Theta_\mu(2B_{Q_{j+1}})> A\,\Theta_\mu(2B_R)$ if $j$ is big enough (depending on $A$ and $\tau$), which
contradicts the fact that $Q_{j+1}\in\tree(R)$.
\end{proof}
\vv

Consider $\ve\in (0,\ve_{0})$ to be chosen later. Set 
\[\TT =\{Q\in \tree(R): Q\supsetneq P\mbox{ for some }P\in (\tree(R)\cap \DD^{db})\backslash \sss(R)\}\]
and set  $\TT^{k}=\TT\cap \DD_{k}(R)$.  For $Q\in \tree(R)$, let $\widehat{Q}\in\tree(R)\setminus \sss(R)$ be the smallest cube in $\TT$ containing $Q$, so that  $\wh Q$ is doubling, balanced, and,  by Lemma \ref{lemdob32}, $\ell(\widehat{Q})\lesssim \ell(Q)$.

Let $C=60\ll \delta^{-1}$. For $Q\in\TT$, pick $x_{Q}\in Q\cap F$ such that 
\[ \int_{C^{-1}r(Q)}^{Cr(Q)} \beta_{\mu,2}(x_{Q},r)^{2}\frac{dr}{r}\leq \beta(Q)^{2}\Theta_\mu(2B_{R})\]
and $\rho(Q)\in [(C-1)r(Q),Cr(Q)]$ such that
\[\beta_{\mu,2}(x_{Q},\rho(Q))\lesssim \beta(Q)\Theta_\mu(2B_{R})^{\frac{1}{2}}\]

Observe, that, since $C=60$, and $x_{Q}\in Q\subset B_Q= 28B(Q)=B(z_{Q},28r(Q))$,
\[ B(Q)=B(z_{Q},r(Q))\subset B(x_{Q},r(Q)+28r(Q))\subset B(x_{Q},\rho(Q))\]
and
\[ B(x_{Q},\rho(Q))\subset B(z_{Q},\rho(Q)+28r(Q))\subset B(z_{Q},88r(Q))\subset 100B(Q).\]
Thus, we always have
\begin{equation}\label{bqbxqrq}
B(Q)\subset B(x_{Q},\rho(Q))\subset 100 B(Q).
\end{equation}

Let $L_{Q}$ be an $n$-plane so that 
\begin{equation}\label{betaxqrq}
 \beta_{\mu,2}(x_{Q},\rho(Q))^{2}=  \int_{B(x_{Q},\rho(Q))} \ps{\frac{\dist(y,L_{Q})}{\rho(Q)}}^{2}\,
\frac{d\mu(y)}{\rho(Q)^{n}}.
 \end{equation}
We now are going to assign to each cube $Q\in \TT$ a point $y_{Q}\in Q$ and an $n$-plane $L_{Q}$ passing through $y_{Q}$. 

\begin{enumerate}
\item[(a)] If $Q=\widehat{Q}$, then $Q$ is doubling, so by \eqref{bqbxqrq}, $\mu(B(x_{Q},\rho(Q)))\sim \mu(B(Q))$. Moreover,
\[ \mu(B(Q))\leq \mu(2B_{Q})\leq \mu(100 B(Q))\lesssim \mu(B(Q)),\]
and by Lemma \ref{lemdobbb1} we also have 
\begin{equation}\label{xqrqtheta1}
\mu(B(x_{Q},\rho(Q)))\sim \mu(2B_{Q})=\Theta_\mu(2B_{Q})\ell(Q)^{n}\sim_{A,\tau} \Theta_\mu(2B_{R})\rho(Q)^{n}
\end{equation}
Thus, by Chebyshev's inequality and \eqref{betaxqrq}, there is $y_{Q}\in B(Q)$ so that 
\[\dist(y_{Q},L_{Q})\lesssim \frac1{\Theta_\mu(2B_{R})^{1/2}}\,\beta_{\mu,2}(x_{Q},\rho(Q)) \rho(Q) \lesssim \beta(Q)\ell(Q).\]
Set $L^{Q}$ be the $n$-plane parallel to $L_{Q}$ containing $y_{Q}$.
\vv

\item[(b)] If $Q\neq \widehat{Q}$, let $Q'\in \tree(R)\backslash \sss(R)$ be a doubling cube properly contained in $Q$ with maximal side length (this exists by our definition of $\TT$). Then 
\[ 2B_{Q'}= B(z_{Q'},56r(Q'))\subset 100 B(Q')\]
and since $Q'$ is doubling,
\[ \mu(2B_{Q'})\leq \mu(100 B(Q'))\lesssim \mu(Q') \leq \mu(2B_{Q'}).\]
Since $z_{Q}\in Q\subset \widehat{Q}\subset B_{Q}$,
\[ 
B(x_{\widehat{Q}},\rho(\widehat{Q}))\subset B(x_{\widehat{Q}},60 r(\widehat{Q}))
\subset B(z_{\widehat{Q}},(28+60)r(\widehat{Q})) \subset 100B(\widehat{Q}).\]
Also, taking into account that $Q',\widehat{Q}\not\in \sss(R)$, 
\[ 
\mu(B(x_{\widehat{Q}},\rho(\widehat{Q})))
\leq \mu(100 B(\widehat{Q}))\lec \mu(\widehat{Q})\leq \mu(2B_{\widehat{Q}}).\]
Since $\ell(Q')\sim_{A,\tau} \ell(Q)\sim_{A,\tau} \ell(\widehat{Q})$ by Lemma \ref{lemdob32}, we have by Lemma \ref{lemdobbb1} and the previous case applied to $\widehat{Q}$,
\begin{align}\label{QQ'Qhat}
\mu(2B_{\widehat{Q}}) 
& \sim_{A,\tau} \mu(2B_{Q'})
\lec \mu(Q')
\leq \mu(Q)
\leq \mu (B(x_{\widehat{Q}},\rho(\widehat{Q})))
\notag \\
&\lec  \mu(2B_{\widehat{Q}})\sim_{A,\tau} \Theta_\mu(B_{R})\ell(\widehat{Q})^{n}  \sim_{A,\tau} \Theta_\mu(2B_{R})\ell(Q)^{n}.\end{align}
Thus, we can use Chebychev's inequality to find $y_{Q}\in Q$ so that 
\[
\dist(y_{Q},L_{\widehat{Q}})\lec_{A,\tau}  \frac{\beta_{\mu,2}(x_{\widehat{Q}},\rho(\widehat{Q}))\rho(\widehat{Q})}{\Theta_\mu(2B_{R})^{\frac{1}{2}}}\lec_{A,\tau}  \beta(\widehat{Q})\ell(\widehat{Q}).\]
We now let $L^{Q}$ be the $n$-plane parallel to $L_{\widehat{Q}}$ but containing $y_{Q}$.  
\end{enumerate}
Observe that, after replacing $L_{Q}$ with $L^{Q}$ in either of these cases, we still have 
\begin{equation}\label{L^Q}
 \int_{B(x_{\widehat{Q}},\rho(\widehat{Q}))} \ps{\frac{\dist(y,L^{Q})}{\rho(\widehat{Q})}}^{2}\,\frac{d\mu(y)}{\rho(\widehat{Q})^n}\lec_{A,\tau} \beta_{\mu,2}(x_{\widehat{Q}},\rho(\widehat{Q}))^{2}\lec_{A,\tau} \beta(\widehat{Q})^{2}\Theta_\mu(2B_{R}).
\end{equation}

We need now to estimate the angles between the $n$-planes $L^{Q}$ corresponding to cubes $Q$ that are near each other. This task is carried out in the next two lemmas. The first one is a well known general result alluded to at the end of Section 5 in \cite{DS1}, without proof. For the reader's convenience, we include a proof in the Appendix.

\begin{lemma}\label{angles}
Suppose $P_{1}$ and $P_{2}$ are $n$-planes in $\R^{d}$ and $X=\{x_{0},...,x_{n}\}$ are points so that
\begin{enumerate}
\item[(a)] $\eta=\eta(X)=\min\{\dist(x_{i},\spn X\backslash\{x_{i}\})/\diam X\in (0,1)$ and
\item[(b)] $\dist(x_{i},P_{j})<\ve\,\diam X$ for $i=0,...,n$ and $j=1,2$, where $\ve<\eta d^{-1}/2$.
\end{enumerate}
Then
\begin{equation}
\dist(y,P_{1}) \leq \ve\ps{\frac{2d}{\eta}\dist(y,X)+\diam X}.
\end{equation}
\end{lemma}

The next lemma tailors the previous one to our setting.

\begin{lemma}\label{distlqlemma}
Suppose $Q_{1},Q_{2}\in \TT\cap \DD^{db}$ are such that $\widehat{Q}_{i}=Q_{i}$ and $\dist(Q_{1},Q_{2})\lesssim \ell(Q_{1})\sim \ell(Q_{2})$. Let $P\in \TT\cap \DD^{db}$ be the smallest cube such that $B(x_{P},\rho(P))\supset B(x_{Q_{1}},\rho(Q_{1}))\cup B(x_{Q_{2}},\rho(Q_{2}))$. Then $\ell(P)\sim \ell(Q_{1})\sim \ell(Q_{2})$ and
\begin{multline}\label{distlq}
\dist(y,L^{Q_{1}})\lesssim_{A,\tau} \beta(P)\bigl(\dist(y,{Q}_{1})+\ell({Q}_{1})+\dist(y,{Q}_{2})+\ell({Q}_{2}))\\
\leq\alpha\,(\dist(y,{Q}_{1})+\ell({Q}_{1})+\dist(y,{Q}_{2})+\ell({Q}_{2})\bigr)\;
 \mbox{ for all }y\in L^{Q_{2}}.
\end{multline}
\end{lemma}

\begin{proof}
Note that since $B(x_{R},\rho(R))\supset B(x_{Q_{1}},\rho(Q_{1}))\cup B(x_{Q_{2}},\rho(Q_{2}))$, Lemma \ref{lemdob32}  implies $P$ is well defined and $B(x_{P},\rho(P))\supset B(x_{Q_{1}},\rho(Q_{1}))\cup B(x_{Q_{2}},\rho(Q_{2}))$. Moreover, observe that $\beta(P)<\alpha$ since $P\not\in \mathsf{J}(R)$.

Let $x_{0},...,x_{n}\in Q_{1}$ be the points from Lemma \ref{lembalan} for the cube $Q=Q_{1}$ with  $\gamma=\tau^2$
and $t=t_0(\gamma)$ (see Remark \ref{remt0g}).
Then by \eqref{xqrqtheta1},
 \begin{align*}
\mu(B(x_{i},t\rho(Q_{1}))) &\geq \ve(t)\, \mu(Q_{1})\\
&\gtrsim_{A,\tau} \mu(B(x_{Q_{1}},\rho(Q_{1})))\sim _{A,\tau}\Theta_\mu(2B_{Q_{1}})\ell(Q_{1})^{n}\sim_{A,\tau} \Theta_\mu(2B_{R})\ell(Q_{1})^{n},
\end{align*}
and so
\begin{align*}
\avint_{B(x_{i},t\rho(Q_{1}))} \ps{\frac{\dist(x,L^{Q_{1}})}{t\rho(Q_{1})}}^{2}d\mu(x)
& \lec_{A,\tau} \int_{B(x_{Q_{1}}, \rho(Q_{1}))}\ps{\frac{\dist(x,L^{Q_{1}})}{\rho(Q_{1})}}^{2} \frac{d\mu(x)}{\Theta_\mu(2B_{R})\rho(Q_{1})^{n}}\\
& \lec_{A,\tau} \frac{\beta_{\mu,2}(x_{Q_{1}},\rho(Q_{1}))^{2}}{\Theta_\mu(2B_{R})}.
\end{align*}
Observe that since $B(x_{Q_{1}},\rho(Q_{1}))\subset B(x_{P},\rho(P))$, we have by \eqref{QQ'Qhat}
\begin{align*}
\beta_{\mu,2}(x_{Q_{1}},\rho(Q_{1}))^{2}
& \leq \avint_{B(x_{Q_{1}}, \rho(Q_{1}))}\ps{\frac{\dist(x,L^{P})}{\rho(Q_{1})^{n+1}}}^2d\mu(x)\\
&\lesssim \avint_{B(x_{P},\rho(P))}\left(\frac{\dist(x,L^{P})}{\rho(P)^{n+1}}\right)^2d\mu(x)
=\beta_{\mu,2}(x_{P},\rho(P))^{2} \lesssim \beta(P)^{2}\Theta_\mu(2B_{R}).
\end{align*}
Thus,
\[
\avint_{B(x_{i},t\rho(Q_{1}))}\ps{\frac{\dist(x,L^{Q_{1}})}{t\rho(Q_{1})}}^{2}d\mu(x)\lesssim_{A,\tau} \beta(P)^{2}.\]
Similarly,
\[
\avint_{B(x_{i},t\rho(Q_{1}))}\ps{\frac{\dist(x,L^{P})}{t\rho(Q_{1})}}^{2}d\mu(x) 
\lec_{A,\tau} \beta(P)^{2}.
\]
Using Chebyshev's inequality, we may find $y_{i}\in B(x_{i},t\rho(Q_{1}))\cap \supp\mu$ such that
\[\max\{\dist(y_{i},L^{Q_{1}}),\dist(y_{i},L^{P})\}\lesssim_{A,\tau} \beta(P).\]
From the definition of $t_0(\gamma)$, we can guarantee that, independently of our choice of $y_{i}\in B(x_{i},tr)$, if $L_{k}^{y}$ denotes the $k$-plane containing $y_{0},...,y_{k}$, then $\dist(y_{k},L_{k-1}^{y})\geq \tau^2 \,r/2$. By Lemma \ref{angles}, it follows that
\[\dist(y,L^P)\lesssim_{A,\tau} \beta(P)(\dist(y,Q_{1})+\ell(Q_{1})) \mbox{ for all }y\in L^{Q_{1}}\]
and
\[\dist(y,L^{Q_{1}})\lesssim_{A,\tau} \beta(P)(\dist(y,Q_{1})+\ell(Q_{1})) \mbox{ for all }y\in L^P.\]
With the roles of $Q_{1}$ and $Q_{2}$ reversed, we also get 
\[\dist(y,L^P)\lesssim_{A,\tau} \beta(P)(\dist(y,Q_{2})+\ell(Q_{2})) \mbox{ for all }y\in L^{Q_{2}}\]
and
\[\dist(y,L^{Q_{2}})\lesssim_{A,\tau} \beta(P)(\dist(y,Q_{2})+\ell(Q_{2})) \mbox{ for all }y\in L^P.\]
Thus, by the triangle inequality, we obtain \eqref{distlq}.
\end{proof}
\vv

For $r_k=10^{-k}$, $k\in\N$, pick $s(k)$ so that $56C_{0}A_{0}^{-s(k)}\leq r_{k}< 56C_{0}A_{0}^{-s(k)+1}$ and let $\{x_{jk}\}_{j\in J_{k}}$ be a maximally $r_{k}$-separated subset of $\{y_{Q}:Q\in \TT^{s(k)}\}$, set $Q_{jk}$ to be the cube in $\TT^{s(k)}$ so that $x_{jk}=y_{Q_{jk}}$, and let $B_{jk}$ be as in Theorem \ref{DT}. \\

We claim that the points $x_{jk}$ satisfy \eqref{V2}. So let $x_{jk}$ be one of our points. If $s(k)=s(k-1)$, then $x_{jk}=y_{Q_{jk}}$ for some $Q_{jk}\in \TT^{s(k)}=\TT^{s(k-1)}$, but since the $x_{i,k-1}$ are a maximal $r_{k-1}$-net for $\{y_{Q}:Q\in \TT^{s(k-1)}\}$ and $s(k)=s(k-1)$, we know $x_{jk}\in B_{i,k-1}$ for some $i\in J_{k-1}$, which finishes this case. If $s(k)<s(k-1)$,  let $P\in \TT^{s(k-1)}$ be the unique (and strictly larger) ancestor of $Q_{jk}$ in $\tree(R)$ (note that we may assume that such an ancestor exists, for otherwise $Q_{jk}=R$, $\{x_{ik}\}_{i\in J_{k}}$ consists only of $y_{R}$, so $x_{jk}=y_{R}$, but moreover, $\{x_{i,k-1}\}_{i\in J_{k-1}}=\{y_{R}\}$, and so we trivially have \eqref{V2}).  Then $y_{P}\in B_{i,k-1}$ for some $i\in J_{k-1}$ since $\{x_{i,k-1}\}_{i\in J_{k-1}}$ is a maximal $r_{k-1}$-net in $\{y_{T}:T\in \DD^{s(k-1)}\}$. Moreover, 
\[\diam P\leq \diam B_{P} \leq 56C_{0}A_{0}^{-s(k-1)}\leq r_{k-1},\]
and since $x_{jk}\in Q_{jk}\subset P$ and $Q_{jk}\cap B_{i,k-1}\neq \emptyset$, the above estimate implies $x_{jk}\in 2B_{i,k-1}$, and this proves the claim.\\

Set $L_{jk}=L^{\widehat{Q}_{jk}}$. In order to apply Theorem \ref{DT}, we need to check that the estimate \eqref{ek} holds. 
For a given $x\in \mathbb{R}^{d}$, fix $k_{0}$ and pick $x_{j_{0}k_{0}}$ so that $x\in 100 B_{j_{0}k_{0}}$, if it exists.

Suppose $x\in 100 B_{jk}\cap 100 B_{lm}$ for some $k\leq k_{0}$, $|k-m|\leq 2$, $j\in J_{k}$, $l\in J_{m}$, and $k\leq m$. Let $Q_{j_{0}k_{0}}^{k}$ denote the ancestor of $Q_{j_{0}k_{0}}$ in $\TT^{s(k)}$, and let $P_{k}\supset Q_{j_{0}k_{0}}^{k}$ be an ancestor that is doubling and such that $B(x_{P_{k}},r_{P_{k}})\supset B(x_{\widehat{Q}_{jk}},r_{\widehat{Q}_{jk}}) \cup B(x_{\widehat{Q}_{lm}},r_{\widehat{Q}_{lm}})$ and $\ell(P_{k})\lec \ell (Q_{j_{0}k_{0}}^{k})\sim \ell(\widehat{Q}_{jk})\sim \ell(\widehat{Q}_{lm})$. By Lemma \ref{distlqlemma}, we have  
\begin{equation}\label{distlqwidehat}
\dist(y,L_{jk})\lesssim_{A,\tau} \beta(P_{k})\,(\dist(y,\widehat{Q}_{jk})+\ell(\widehat{Q}_{jk})+\dist(y,\widehat{Q}_{lm})+\ell(\widehat{Q}_{lm}))\quad \mbox{for all }y\in L^{Q_{lm}},
\end{equation}
as well as the same inequality if we trade the roles of $\widehat{Q}_{jk}$ and $\widehat{Q}_{lm}$. Note that $\widehat{Q}_{jk}$ and $\widehat{Q}_{lm}$ are at a distance at most $100r_{k}$ from $x$ and have side lengths comparable to $r_{k}$, hence
\[\dist(y,L_{jk})\lesssim_{A,\tau} \beta(P_{k})(|y-x|+r_{k})\quad \mbox{\,for all\,}\quad y\in L_{lm}\]
and from this it is not difficult to show
\[d_{x_{jk},10^{4}r_{l}}(L_{jk},L_{lm})\lesssim_{A,\tau} \beta(P_{k}).\]
Taking the maximum over all $x_{jk}$ and $x_{ml}$ with $x\in 100 B_{jk}\cap 100 B_{lm}$, $|k-m|\leq 2$, $j\in J_{k}$, $l\in J_{m}$, and $m\geq k$ (we let $k$ stay fixed), we get $\ve_{k}(x)\lesssim_{A,\tau} \beta(P_{k})$.

Note that for any cube $P$ there can be at most a bounded number of cubes $P_k$ (depending on $A_{0}$ and $C_{0}$) for which $P_{k}=P$. Therefore, since $\TT$ contains no cubes in $\mathsf{J}(R)$,
\[\sum_{k=0}^{k_{0}} \ve_{k}(x)^{2}\lesssim_{A,\tau} \sum_{Q_{j_{0}k_{0}}\subset P\subset R}\beta(P)^{2}< \alpha^{2},\]
and since $k_{0}$ is arbitrary, we also get $\sum_{k=0}^{\infty} \ve_{k}(x)^{2}\lesssim_{A,\tau} \alpha^{2}$.
Hence, for $\alpha>0$ small enough, this sum is less than $\ve$ and \rf{ek} is fulfilled. \\

Now we can apply Theorem \ref{DT} to obtain an $L$-bi-Lipschitz homeomorphism $g:\mathbb{R}^{n}\rightarrow \mathbb{R}^{d}$,  where $L$ is a universal constant, so that
the set $E_\infty$ from \rf{eqahj0} is contained in $g(\R^n)$ and \rf{closetog} holds.
 Set $\Gamma_R=g(\mathbb{R}^{d})$. Note that if $x\in F$ is not contained in a cube from $\sss(R)$, then it is contained in infinitely many cubes from $\tree(R)$ and hence infinitely many cubes from $\TT$. Thus, we can write $x$ as the limit of a sequence $y_{Q_{k}}$ where $x\in Q_{k}\in \TT^{k}$, and $y_{Q_{k}}\in B_{j_{k}x_{k}}$ for some $j_{k}\in J_{k}$. Therefore, we can write $x$ as a limit of the form $x=\lim x_{j(k),k}$ for some $j(k)\in J_{k}$, which implies $F\subset E_{\infty}\subset \Gamma_R$.\vv

\subsection{The small measure of the cubes from $\LD(R)$}
All that is left to do now now is control the measure of the low-density cubes. To this end, we will show first that most of the
measure of $R$ lies close to the surface $\Gamma_R$. To the authors' surprise, 
the arguments below work with $\beta_{\mu,2}$ but not with $\beta_{\mu,p}$ with $p<2$.
This seems to indicate a subtle difference between these coefficients.

Let 
\[ \mathsf{Far}=\bigl\{x\in R: \dist(x,L^{Q})\geq \alpha^{1/2}\ell(Q) \mbox{ for some }Q\in \TT\cap\DD^{db}\bigr\}.\]
By Chebyshev's inequality we have
\[
\alpha^{1/2}\,\mu(\mathsf{Far})
 \leq \int_{R}\ps{\sum_{x\in Q\in \TT\cap\DD^{db}}\ps{\frac{\dist(x,L^{Q})}{\ell(Q)}}^{2}}^{\frac{1}{2}}d\mu(x).\]
 By Cauchy-Schwarz, the right hand side is at most
\[ \ps{\int_{R}\sum_{Q\in \TT\cap\DD^{db}}\ps{\frac{\dist(x,L_{Q})}{\ell(Q)}}^{2}d\mu(x)}^{\frac{1}{2}}\mu(R)^{\frac{1}{2}}.\]
Since $\ell(Q)\sim_{A,\tau} \rho(Q)$ and $\mu(2B_{Q})=\Theta_\mu(2B_{Q})r(2B_{Q})^n\sim_{A,\tau} \Theta_\mu(2B_{R})\rho(Q)^{n}$ by Lemma \ref{lemdobbb1}, we get that the above does not exceed
\[
c(A,\tau) \ps{\sum_{Q\in \TT \cap\DD^{db}}\int_{B(x_{Q},\rho(Q))}\ps{\frac{\dist(x,L^{Q})}{r_Q}}^{2}\frac{d\mu(x)}{\rho(Q)^{n}} \,\frac{\mu(2B_{Q})}{\Theta_\mu(2B_{R}) }}\mu(R)^{\frac{1}{2}}.\]
By \eqref{L^Q} the last integral does not exceed $c(A,\tau)\,\beta(Q)^2\Theta_\mu(2B_{R})$, and so the above inequalities imply 
 \begin{align*}
\alpha^{1/2}\,\mu(\mathsf{Far}) 
&  \lesssim_{A,\tau}\ps{\sum_{Q\in \TT\cap\DD^{db}}\beta(Q)^{2} \mu(2B_{Q}) }^{\frac{1}{2}} \mu(R)^{\frac{1}{2}}
 \lesssim_{A,\tau} \ps{\sum_{Q\in \TT\cap\DD^{db}}\beta(Q)^{2} \mu(Q)}^{\frac{1}{2}}\mu(R)^{\frac{1}{2}}\\
 & =\ps{\int_{R}\sum_{x\in Q\in \TT\cap\DD^{db}}\beta(Q)^{2}d\mu(x)}^{\frac{1}{2}}\mu(R)^{\frac{1}{2}}
  \leq\alpha\,\mu(R),
\end{align*}
where in the last inequality we used the fact that no cube in $\TT$ is in $\mathsf{J}(R)$. Thus, for $\alpha$ small enough (depending on $A,\tau$), we have 
\[\mu(\mathsf{Far})\leq\frac12\,\tau^{\frac{1}{2}}\mu(R).\]

So to prove \eqref{sumld} it suffices to show
\begin{equation}
\sum_{Q\in \LD_{close}(R)\neq\emptyset}\mu(Q)\leq \frac12\,\tau^{\frac{1}{2}}\mu(R), \;\; \mbox{ where }\LD_{close}(R)=\{Q\in \LD(R): Q\backslash \mathsf{Far}\neq\emptyset\}.
\label{ldnotfar}
\end{equation}
We claim that it suffices to prove that for each $Q\in \LD_{close}(R)$ there is a point
\begin{equation}\label{closetogamma}
 \xi_{Q}\in \frac{3}{2}B_{Q}\cap \Gamma_R.
 \end{equation}
Assuming this for a moment, let us finish the proof of the theorem. By the Besicovitch covering theorem, there are cubes $Q_{j}\in \LD_{close}(R)$ so that $\bigcup_{Q\in \LD_{close}(R)}2B_{Q}\subset \bigcup_j 2B_{Q_{j}}$ and so that no point is contained in at most $N=N(d)$ many $2B_{Q_{j}}$. Moreover, since $R\not\in \LD(R)$, we know that each $Q\in \LD(R)$ is such that $r(Q)\leq C_{0}A_{0}^{-1}r(R)$, and thus $2B_{Q}\subset 2B_{R}$ for $A_{0}$ large enough. Since $\Gamma_R= g(\R^{n})$ where $g$ is $L$-bi-Lipschitz and $L$ depends only on $n$ and $d$, we know $\Gamma_R$ is Ahlfors regular. Using these facts and that $R\in \DD^{db}$, $\Theta_\mu(2B_{Q})\leq \tau \,\Theta_\mu(2B_{R})$ for $Q\in \LD(R)$, and $B(\xi_{Q},r(B_{Q})/2)\subset 2B_{Q}$, we obtain
\begin{align*}
\sum_{Q\in \LD_{close}(R)}\mu(Q)
& \leq 
\sum_j \mu(2B_{Q_{j}})
=\sum_j  \Theta_\mu(2B_{Q_{j}})r(2B_{Q_{j}})^{n}\\
&  \lesssim \tau\Theta_\mu(2B_{R})\sum_j \cH^{n}(\Gamma_R\cap B(\xi_{Q_{j}},r(B_{Q_{j}})/2))
\leq \tau\Theta_\mu(2B_{R})\sum_j\cH^{n}(\Gamma_R\cap 2B_{Q_{j}})\\
& \lesssim \tau\Theta_\mu(2B_{R})\cH^{n}\Bigl(\Gamma_R\cap \bigcup_j 2B_{Q_{j}}\Bigr)
\leq \tau\Theta_\mu(2B_{R})\cH^{n}(\Gamma_R\cap 2B_{R})\\
& \lesssim \tau\Theta_\mu(2B_{R})r(2B_{R})^{n}
\sim\tau \mu(2B_{R})\lesssim \tau \mu(R)
\end{align*}
and so for $\tau$ small enough we have \eqref{ldnotfar}. 

We now focus on showing \eqref{closetogamma}. The main idea is that we know if $Q\in \LD_{close}(R)$, there is $x$ close to $L^{\widehat{Q}}$ (from the definition of $\LD_{close}(R)$). We would like to use \eqref{closetog} to conclude that $x$ is close to $\Gamma_R$ and hence we can find an appropriate $\xi_{Q}$, but we can only use that inequality if $L^{\widehat{Q}}$ happens to be one of the $L_{jk}$ we used to apply Theorem \ref{DT}. However, we can still find a cube $Q_{jk}$ of size and distance from $\widehat{Q}$ comparable to $\ell(Q)$, and by our work above we know that the distance between the planes $L_{jk}$ and $L^{\widehat{Q}}$ is small. Thus, $x$ is close to a point $y\in L^{\widehat{Q}}$, which is close to a point $z\in L_{jk}$ which, by \eqref{closetog}, is close to a point $\xi_{Q}\in \Gamma_R$ Now we will provide the details.

For a given $Q\in \LD_{close}(R)$ there exists $x\in Q$ such that $\dist(x,L^{\widehat{Q}})\leq\alpha^{1/2}\,\ell(\widehat{Q})\lesssim_{A,\tau} \alpha^{1/2}\ell(Q)$. Let $y\in L^{\widehat{Q}}$ be the projection of $x$ onto $L^{\widehat{Q}}$, so $y\in \frac{5}{4}B_{Q}$ for $\alpha>0$ small enough. Pick $k$ so that $\widehat{Q}\in \DD_{s(k)}$, thus $r(B_{\widehat{Q}})\leq r_{k}$. Then there is $Q_{jk}$ with $y_{\widehat{Q}}\in B_{jk}$, and so 
\[\dist(\widehat{Q},Q_{jk})\leq |y_{\widehat{Q}}-x_{jk}|\leq r_{k}\sim \ell(Q_{jk})\sim_{A,\tau} \ell(\widehat{Q}).\]
Thus, we can use Lemma \ref{distlqlemma} and the fact that $y\in L^{\widehat{Q}}\cap \frac{5}{4}B_{Q}$ to conclude
\[
\dist(y,L_{jk})\lesssim \alpha(\dist(y,\widehat{Q})+\ell(\widehat{Q})+\dist(y,Q_{jk})+\ell(Q_{jk}))
\lesssim_{A,\tau} \alpha\, \ell(Q).
\]
Let $z$ be the projection of $y$ onto $L_{jk}$, so by the above inequality, $|z-y|\lesssim_{A,\tau} \alpha\ell(Q)$. Thus, this inequality, our definition of $y$, and the fact that $x\in Q$ imply
\[|z-x_{jk}|
\leq |z-y|+|y-x|+|x-y_{\widehat{Q}}|+|y_{\widehat{Q}}-x_{jk}|
\lesssim_{A,\tau} \alpha\ell(Q)+\alpha^{1/2}\ell(Q)+r(B_{\widehat{Q}})+r_{k}
<3r_{k}\]
if $\alpha$ is small enough. Thus, by \eqref{closetog} in Theorem \ref{DT}, $\dist(z,\Gamma_R)\lesssim \ve r_{k}\sim \ve \ell(Q)$, and thus 
\[\dist(x,\Gamma_R)\leq |x-y|+|y-z|+\dist(z,\Gamma_R)\lesssim \alpha^{1/2}\ell(Q)+\alpha\ell(Q) +\ve\ell(Q)<\frac{3}{2}r(B_{Q})\]
if $\alpha$ and $\ve$ are chosen small enough. Thus, we can find $\xi_{Q}\in \Gamma_R\cap B(x,\frac{3}{2}r(B_{Q}))$, which proves \eqref{closetogamma}.
\vv


\section{The $\beta_2$'s and Menger curvature, and further remarks} \label{secrem}
\vv

By arguments analogous to the ones used to prove Theorem \ref{teo1} one also gets the following:
 
\begin{theorem}\label{teo1*}
Let $p\geq 0$ and let $\mu$ be a finite Borel measure in $\R^d$ 
such that $0<\Theta^{n,*}(x,\mu)<\infty$ for $\mu$-a.e.\ $x\in\R^d$. If
\begin{equation}\label{eqjones**}
\int_0^1 \beta_{\mu,2}(x,r)^2\,\Theta_\mu(x,r)^p\,\frac{dr}r<\infty \quad\mbox{ for $\mu$-a.e.\ $x\in\R^d$,}
\end{equation}
then $\mu$ is $n$-rectifiable.
\end{theorem}

Clearly, since $\Theta^{n,*}(x,\mu)<\infty$, the larger is $p$, the weaker is the assumption \rf{eqjones**}.

We will now sketch the required changes to obtain this result.
First, it is easy to check that in Main Lemma \ref{mainlemma} one can replace the inequality in (c) by
\begin{multline*}
\sum_{Q\in\sss(R)\setminus (\HD(R)\cup\UB(R)) }\!\!\!\!\mu(Q)\\ \leq \tau^{1/2}\,\mu(R) 
+ 
\frac{c(A,\tau)}{\Theta_\mu(2B_R)^{p+1}}\sum_{Q\in\tree(R)} \int_{F\cap\delta^{-1}B_Q}\int_{\delta\ell(Q)}^{\delta^{-1}\ell(Q)}
\beta_{\mu,2}(x,r)^2\,\Theta_\mu(x,r)^p\,\frac{dr}r\,d\mu(x).
\end{multline*}
Using this estimate and arguments analogous to the ones in Lemma \ref{lemkey62}, one deduces the following:

\begin{lemma}\label{lemkey62*}
If $\tau$ is chosen small enough in the Main Lemma, then
\begin{equation}\label{eqsum441*}
\sum_{R\in\ttt}\Theta_\mu(2B_R)^{p+1}\,\mu(R)\leq C_*^{p+1}\,\mu(R_0)+
c(A,\tau,\eta,\delta)\,\int_F\int_0^1\beta_{\mu,2}(x,r)^2\,\Theta_\mu(x,r)^p\frac{dr}r\,d\mu(x),
\end{equation}
where $C_*$ is the constant in \rf{eqc*}.
\end{lemma}

With this result at hand,  using that 
 for $\mu$-a.e.\ $x\in R_0$,
\begin{equation*}
\sum_{R\in\ttt:x\in R} \Theta_\mu(2B_R)^{p+1}<\infty
\end{equation*}
instead of \rf{eqsug9}, the same arguments as in Subsection \ref{subhh*} show that $\mu$ is $n$-rectifiable.
\vv

The case $p=1$ of Theorem \ref{teo1*} is particularly interesting because of the relationship with the Menger curvature of measures and singular integrals
due to the estimate \rf{eqsum441*} and the results in Sections 17 and 19 in \cite{Tolsa-delta}. \\

Our goal now is to prove Theorem \ref{teocurv-intro}, which we state below again for the reader's convenience. 

\begin{theorem*}  
Let $\mu$ be a finite Radon measure in $\R^2$ such that $\mu(B(x,r))\leq r$ for all $x\in\R^2$. Then
$$c^2(\mu) + \|\mu\|\sim \iint_0^\infty \beta_{\mu,2}(x,r)^2\,\Theta_\mu(x,r)\,\frac{dr}r\,d\mu(x) + \|\mu\|,$$
where the implicit constant is an absolute constant.
\end{theorem*}

Consider the family $\ttt$ defined in Section \ref{secprovam}, with $R_0=F=\supp\mu$ and $\mathcal B=\varnothing$. 
Arguing as in Lemma 17.6 of \cite{Tolsa-delta}, one deduces that if $\mu(B(x,r))\leq r$ for all $x\in\R^2$, then 
$$c^2(\mu) \lesssim \sum_{R\in \ttt}\Theta_\mu(2B_R)^2\,\mu(R).$$
Combining this estimate with Lemma \ref{lemkey62*} (with $n=p=1$), we obtain
$$c^2(\mu)\lesssim \|\mu\| + \iint_0^\infty \beta_{\mu,2}(x,r)^2\,\Theta_\mu(x,r)\,\frac{dr}r\,d\mu(x).$$

To complete the proof of Theorem \ref{teocurv-intro} it remains to prove the converse by showing that
\begin{equation}\label{eqrem77}
\iint_0^\infty \beta_{\mu,2}(x,r)^2\,\Theta_\mu(x,r)\,\frac{dr}r\,d\mu(x)\lesssim c^2(\mu) + \|\mu\|.
\end{equation}
To this end, 
 we will use the corona decomposition of \cite{Tolsa-bilip}. To describe it we follow quite closely the approach in 
\cite[Section 19]{Tolsa-delta}.
To state the precise result we need, first we will introduce some terminology which is quite similar to the one 
of the corona construction in Section \ref{secprovam}. An important difference is that it involves the usual dyadic lattice $\DD(\R^2)$, instead
of the David-Mattila lattice $\DD$. 

Let $\mu$ be a finite Radon measure, and assume that there exists a dyadic square $R_0\in\DD(\R^2)$
such that $\supp\mu\subset R_0$ with $\ell(R_0)\leq 10\,\diam(\supp(\mu))$, say. 
Let $\ttt_{*}\subset
\DD(\R^2)$ be
a family of dyadic squares contained in $R_0$, with $R_0\in\ttt_{*}$. 

 Given $R\in\ttt_{*}$, we denote by
 $\eend_{*}(R)$ 
the subfamily of the squares
$P\in \ttt_*$ satisfying
\begin{itemize}
\item $P\subsetneq R$,
\item $P$ is maximal, in the sense that there does not exist
another square $P'\in \ttt_*$ such that $P\subset P'\subsetneq R$.
\end{itemize}
 Also, we denote by $\tr_*(R)$ the family of squares $\DD(\R^2)$ which intersect $\supp\mu$, are contained in
$R$, and are not contained in any square from $\eend_*(R)$. 
Notice that 
$$\{P\in\DD(\R^2):P\subset R_0,\,P\cap\supp\mu\neq\varnothing\} = \bigcup_{R\in \ttt_*} \tr_*(R).$$
The set of good points contained in $R$ equals
$$G_*(R):= R\cap\supp(\mu)\setminus \bigcup_{P\in\eend_*(R)}P.$$

Given a square $Q\subset\R^2$, we denote
$$\Theta_\mu(Q)= \frac{\mu(Q)}{\ell(Q)},$$
and given two squares $Q\subset R$, we set
$$\delta_{\mu}(Q,R) := \int_{2R\setminus Q} \frac1{|y-z_Q|}\,d\mu(y),$$
where $z_Q$ stands for the center of $Q$. We also set
$$\beta_{\mu,2}(Q) = \inf_L \left(\frac1{\ell(Q)} \int_{Q} \left(\frac{\dist(y,L)}{\ell(Q)}\right)^2\,d\mu(y)\right)^{1/2},$$
where the infimum is taken over all the lines $L\subset\R^2$.
\vv

\begin{lemma}[The dyadic corona decomposition of \cite{Tolsa-bilip}] \label{lemcorona3}
Let $\mu$ be a Radon measure on $\R^2$ with linear growth and finite curvature $c^2(\mu)$. Suppose that
there exists a dyadic square $R_0\in\DD(\R^2)$
such that $\supp\mu\subset R_0$ with $\ell(R_0)\leq 10\,\diam(\supp(\mu))$. 
Then there exists a family $\ttt_*$ as above which satisfies the following. 
For each square $R\in \ttt_*$ there exists an
 AD-regular curve $\Gamma_R$ (with the AD-regularity constant uniformly bounded by some absolute constant) such that:
\begin{itemize}
\item[(a)] $\mu$ almost all $G_*(R)$ is contained in $\Gamma_R$.

\item[(b)] For each $P\in \eend_*(R)$ there exists some square
$\wt{P}\in\DD(\R^2)$ containing $P$, concentric with $P$, such that $\delta_{\mu}(P,\wt{P})\leq
C\Theta_\mu(7R)$ and $\frac12{\wt{P}}\cap \Gamma_R\neq \varnothing$.

\item[(c)] If $P\in\tr_*(Q)$, then $\Theta_\mu(7P)\leq
C\,\Theta_\mu(7R).$
\end{itemize}
Further, the following packing condition holds:
 \begin{equation} \label{pack3}
\sum_{R\in \ttt_*} \Theta_\mu(7R)^2 \mu(7R) \leq C\,\|\mu\| + C\,c^2(\mu).
\end{equation}
\end{lemma}

\vv

Let us remark that the squares from the family $\ttt_*$ may be non-doubling.

The preceding lemma is not stated explicitly in \cite{Tolsa-bilip}. However it follows 
immediately from the Main Lemma 3.1 of \cite{Tolsa-bilip}, just by splitting the so called $4$-dyadic squares in 
 \cite[Lemma 3.1]{Tolsa-bilip} into dyadic squares. Further, the family $\ttt_*$ above is the same as the family
 $\ttt_{\rm dy}$ from \cite[Section 8.2]{Tolsa-bilip}.

We need a couple of auxiliary results from \cite{Tolsa-bilip}. The first one,  introduces a regularized version of the family $\eend_*(R)$ for $R\in\ttt_*$ and is proved in Lemmas 8.2 and 8.3 of \cite{Tolsa-bilip}.
\vv

\begin{lemma}\label{lemreg*}
Let $\ttt_*$ be as in Lemma \ref{lemcorona3}.
For each $R\in\ttt_*$ there exists a family of dyadic squares $\reg_*(R)$ which satisfies the following
properties:
\begin{itemize}
\item[(a)] The squares from $\reg_*(R)$ are contained in $R$ and are pairwise disjoint.

\item[(b)] Every square from $\reg_*(R)$ is contained in some square from $\eend_*(R)$.

\item[(c)] $\bigcup_{Q\in\reg_*(R)}2Q \subset \R^2\setminus G_*(R)$ and
$\supp\mu\cap R\setminus\bigcup_{Q\in\reg_*(R)}Q\subset G_*(R)\subset \Gamma_R$.

\item[(d)] If $P,Q\in \reg_*(R)$ and $2P\cap 2Q\neq \varnothing$, then $\ell(Q)/2 \leq \ell(P) \leq 2 \ell(Q)$.
 
\item[(e)] If $Q\in\reg_*(R)$ and $x\in Q$, $r\geq \ell(Q)$, then $\mu(B(x,r)\cap
4R) \leq C\Theta_\mu(7R)\,r.$

\item[(f)] For each $Q\in \reg_*(R)$, there exists some square $\wt{Q}$, concentric with $Q$,
which contains $Q$, such that $\delta_{\mu}(Q,\wt{Q})\leq
C\Theta_{\mu}(7R)$ and $\frac12\wt{Q}\cap\Gamma_R\neq
\varnothing$.
\end{itemize}
\end{lemma}

We denote by $\treg(R)$ the family of squares $\DD(\R^2)$ which intersect $\supp\mu$, are contained in
$R$, and are not contained in any square from $\reg_*(R)$. Clearly, we have $$\tr_*(R)\subset \treg(R).$$

The second auxiliary result shows how, in a sense, the measure $\mu$  can be approximated on each tree $\treg(R)$ by another measure supported on 
$\Gamma_R$ which is absolutely continuous with respect to length. This is proved in Lemma 
8.4 of \cite{Tolsa-bilip}.

\begin{lemma} \label{repart}
For $R\in\ttt_*$, denote
$\reg_*(R)=:\{P_i\}_{i\geq1}$. For each
$i$, let $\wt{P}_i\in\DD(\R^2)$ be a square containing $P_i$ such that
$\delta_{\mu}(P_i,\wt{P}_i)\leq C\Theta_\sigma(7R)$ and $\frac12 \wt{P}_i\cap \Gamma_{R}\neq
\varnothing$, with minimal side length  (as in (e) of Lemma \ref{lemreg*}).
For each $i\geq1$ there exists some function $g_i\geq0$ supported
on $\Gamma_R\cap \wt{P}_i$ such that
\begin{equation} \label{co1}
\int_{\Gamma_R} g_i\,d\HH^1 = \mu(P_i),
\end{equation}
\begin{equation} \label{co2}
\sum_i g_i \lesssim \Theta_\mu(7R),
\end{equation}
and
\begin{equation} \label{co3}
\|g_i\|_\infty \,\ell(\wt{P}_i) \lesssim\mu(P_i).
\end{equation}
\end{lemma}

\vv
 

\vv
\begin{proof}[\bf Proof of \rf{eqrem77}]
We will show that
\begin{equation}\label{eqff89}
\sum_{Q\in\DD(\R^2):Q\subset R_0} \beta_{\mu,2}(3Q)^2\,\Theta_\mu(3Q)\,\mu(Q)\lesssim
 c^2(\mu) + \|\mu\|,
 \end{equation}
 which is easily seen to be equivalent to \rf{eqrem77}. We consider the corona decomposition of $\mu$ given by Lemma
 \ref{lemcorona3}. By the packing condition \rf{pack3}, to prove \rf{eqff89} it suffices to show that for every $R\in\ttt_*$,
$$\sum_{Q\in \treg(R)} \beta_{\mu,2}(3Q)^2\,\Theta_\mu(3Q)\,\mu(Q) \lesssim 
\Theta_\mu(7R)^2 \mu(7R).
$$
Since $\Theta_\mu(3Q) \lesssim \Theta_\mu(7Q)\lesssim\Theta_\mu(7R)$, it is enough to prove that
\begin{equation}\label{eqff90}
\sum_{Q\in \treg(R)} \beta_{\mu,2}(3Q)^2\,\mu(Q) \lesssim 
\Theta_\mu(7R) \mu(7R).
\end{equation}

Let $\Omega_R =\R^2\setminus \Gamma_R$, and consider the following family of Whitney squares in $\Omega_R$:
we let $\WW(\Omega_R)$ be the set of maximal dyadic squares $Q\subset \Omega_R$ such that $15 Q\cap \Gamma_R=\varnothing$. These squares have disjoint interiors and can be easily shown to satisfy the following properties:
\begin{enumerate}
\item[(a)] $7\ell(Q)\leq \dist(x,\Omega_R^{c})\leq 16\,\diam (Q)$ for all $x\in Q$,
\item[(b)] If $Q,Q'\in \WW(\Omega_R)$ and $3 Q\cap 3 Q'\neq\varnothing$, then $\ell(Q)\sim\ell(Q')$.

\end{enumerate}

We now split the family $\treg(R)$ into two subfamilies:
$\treg_{small}$ and $\treg_{big}(R)$. The former is made up of the squares from $\treg(R)$ which are contained in some square from
$\WW(\Omega_R)$, while  $\treg_{big}(R)=\treg(R)\setminus \treg_{small}(R)$. That is, $\treg_{big}(R)$ consists of the squares $Q\in \treg(R)$ which are not contained in any square from
$\WW(\Omega_R)$. 

First we will deal with the sum associated with the squares from $\treg_{small}(R)$. We set
\begin{align*}
\sum_{Q\in \treg_{small}(R)} \beta_{\mu,2}(3Q)^2\,\mu(Q) &= \sum_{S\in\WW(\Omega_R)} \,\sum_{Q\in \treg(R):Q\subset S}\beta_{\mu,2}(3Q)^2\,\mu(Q)\\
&=
\sum_{S\in\WW(\Omega_R)} \,\sum_{i:P_i\subset S}\, \sum_{Q\in \treg(R): P_i\subset Q\subset S}
\beta_{\mu,2}(3Q)^2\,\mu(P_i).
\end{align*}
For $Q$ as above we use the trivial estimate $\beta_{\mu,2}(3Q)^2\lesssim \Theta_\mu(3Q)$, and then we obtain
$$\sum_{Q\in \treg_{small}(R)} \beta_{\mu,2}(3Q)^2\,\mu(Q) \lesssim \sum_{S\in\WW(\Omega_R)} \,\sum_{i:P_i\subset S}\mu(P_i)\!\sum_{Q\in \treg(R): P_i\subset Q\subset S}\Theta_\mu(3Q).$$
Let $P_i\subset S$ with $S\in\WW(\Omega_R)$. From the definitions of $\wt P_i$ and of the Whitney squares, we deduce that
$S\subset c\wt P_i$, where $c$ is some absolute constant. 
In fact, we may assume without loss of generality that
$$
\ell(\wt P_i)\sim \max\bigl(\ell(S),\,\ell(P_i)\bigr),
$$
since otherwise we may replace $\wt P_i$ by a small concentric cube which does the job (i.e.\ so that 
both (f) from Lemma \rf{lemreg*} and the above estimate hold).
So we easily infer that
\begin{align*}
\sum_{Q\in \treg(R): P_i\subset Q\subset S}\Theta_\mu(3Q) & \lesssim \delta_\mu(P_i,c\wt P_i) + \sup_{Q\in \treg(R): P_i\subset Q\subset S}\Theta_\mu(3Q)\\
& \lesssim 
\delta_\mu(P_i,\wt P_i) +
\Theta_\mu(7R)\lesssim \Theta_\mu(7R).
\end{align*}
Then we get
\begin{equation}\label{eqff110}
\sum_{Q\in \treg_{small}(R)} \beta_{\mu,2}(3Q)^2\,\mu(Q) \lesssim \sum_{S\in\WW(\Omega_R)}  \sum_{i:P_i\subset S}\Theta_\mu(7R)
\mu(P_i) \sim \Theta_\mu(7R)\,\mu(R).
\end{equation}

We turn now our attention to the sum corresponding to the squares $Q\in\treg_{big}(R)$.
For such a square $Q$ and a line $L_Q$ to be chosen below we write
\begin{align}\label{eqff101}
\beta_{\mu,2}(3Q)^2 \,\ell(3Q)& \leq \sum_{i:P_i\cap 3Q\neq \varnothing} \int_{P_i}
\left(\frac{\dist(x,L_{Q})}{\ell(Q)}\right)^2\,d\mu(x) + \int_{3Q\cap G_*(R)}
\left(\frac{\dist(x,L_{Q})}{\ell(Q)}\right)^2\,d\mu(x)
\nonumber\\
& =
\sum_{i:P_i\cap 3Q\neq \varnothing} \int \left(\frac{\dist(x,L_{Q})}{\ell(Q)}\right)^2\,g_i(x)\,d\HH^1|_{\Gamma_R}(x)\nonumber\\
&\quad + 
\sum_{i:P_i\cap 3Q\neq \varnothing} \int \left(\frac{\dist(x,L_{Q})}{\ell(Q)}\right)^2\,\bigl(\chi_{P_i}(x)\,d\mu(x)- 
g_i(x)\,d\HH^1|_{\Gamma_R}(x)\bigr)\nonumber \\
&\quad + \int_{3Q\cap G_*(R)}
\left(\frac{\dist(x,L_{Q})}{\ell(Q)}\right)^2\,d\mu(x) \nonumber\\
&=: I_1 + I_2+ I_3.
\end{align}

We claim now that, for $Q\in\treg_{big}(R)$,
\begin{equation}\label{eqclaim39}
\mbox{if \;$P_i\cap 3Q\neq \varnothing$, \;then \;$\wt P_i\subset c_3Q$,}
\end{equation}
 for some absolute constant $c_3>1$.
 For the moment we assume this to hold and we continue with the proof of \rf{eqrem77}.

We choose $L_Q$ as some line which minimizes $\beta_{\mu,2}(c_3Q)$.
To estimate the term $I_1$ on the right hand side of \rf{eqff101} we use that 
$\sum_i g_i \lesssim \Theta_\mu(7R)$ by \rf{co2}, and that $\supp g_i\subset\wt P_i\subset c_3Q$ (for $i$ such that 
$P_i\cap 3Q\neq \varnothing$). Then we get
\begin{align*}
\sum_{i:P_i\cap 3Q\neq \varnothing} \int \left(\frac{\dist(x,L_{3Q})}{\ell(Q)}\right)^2\,g_i(x)\,d\HH^1|_{\Gamma_R}(x)&
\lesssim \Theta_\mu(7R) \int_{c_3Q} \left(\frac{\dist(x,L_{3Q})}{\ell(Q)}\right)^2\,d\HH^1|_{\Gamma_R}(x)
\\
& \sim \Theta_\mu(7R)\,\beta_{\HH^1|_{\Gamma_R},2}(c_3Q)^2\,\ell(c_3Q).
\end{align*}

To deal with $I_3$ recall that $\mu|_{G_*(R)}$ is absolutely continuous with respect to $\HH^1|\Gamma_R$ with density not exceeding
$c\,\Theta_\mu(7R)$. So we have
$$I_3\lesssim \Theta_\mu(7R) \int_{3Q} \left(\frac{\dist(x,L_{3Q})}{\ell(Q)}\right)^2\,d\HH^1|_{\Gamma_R}(x)
\lesssim \Theta_\mu(7R)\,\beta_{\HH^1|_{\Gamma_R},2}(c_3Q)^2\,\ell(c_3Q).
$$

We consider now the term $I_2$ on the right hand side of \rf{eqff101}. For $i$ such that $P_i\cap 3Q\neq \varnothing$,
we take into account that $\int g_i(x)\,d\HH^1|_{\Gamma_R}(x) = \mu(P_i)$, and then we derive 
\begin{multline*}
 \int \left(\frac{\dist(x,L_{Q})}{\ell(Q)}\right)^2\,\bigl(\chi_{P_i}(x)\,d\mu(x)- g_i(x)\,d\HH^1|_{\Gamma_R}(x)\bigr)\\ = 
\int \left[\left(\frac{\dist(x,L_{Q})}{\ell(Q)}\right)^2 - \left(\frac{\dist(z_{P_i},L_{Q})}{\ell(Q)}\right)^2\right]\,\bigl( 
\chi_{P_i}(x)\,d\mu(x)- g_i(x)\,d\HH^1|_{\Gamma_R}(x)\bigr),
\end{multline*}
where $z_{P_i}$ is the center of $P_i$. Notice that for $x\in\supp(g_i\,\HH^1|_{\Gamma_R} - 
\chi_{P_i}\,\mu)\subset \wt P_i$,
$$\left|\left(\frac{\dist(x,L_{Q})}{\ell(Q)}\right)^2 - \left(\frac{\dist(z_{P_i},L_{Q})}{\ell(Q)}\right)^2\right|
\leq \frac{|x-z_{P_i}|}{\ell(Q)} \cdot\frac{\dist(x,L_{Q}) + \dist(z_{P_i},L_{Q})}{\ell(Q)}\lesssim \frac{\ell(\wt P_i)}{\ell(Q)}.$$
Thus,
$$ \left|\int \left(\frac{\dist(x,L_{Q})}{\ell(Q)}\right)^2\,\bigl(g_i(x)\,d\HH^1|_{\Gamma_R}(x) - 
\chi_{P_i}(x)\,d\mu(x)\bigr)\right|\lesssim
\frac{\ell(\wt P_i)}{\ell(Q)}\,\mu(P_i).$$
So we get
$$I_2\lesssim \sum_{i:P_i\subset c_3Q}
\frac{\ell(\wt P_i)}{\ell(Q)}\,\mu(P_i).$$

From \rf{eqff101} and the estimates we got for $I_1$, $I_2$ and $I_3$ we deduce
$$\beta_{\mu,2}(3Q)^2 \lesssim \Theta_\mu(7R)\,\beta_{\HH^1|_{\Gamma_R},2}(c_3Q)^2 + 
\sum_{i:P_i\subset c_3Q}
\frac{\ell(\wt P_i)}{\ell(Q)^2}\,\mu(P_i).$$
Therefore,
\begin{align}\label{eqff104}
\sum_{Q\in \treg_{big}(R)} \beta_{\mu,2}(3Q)^2\,\mu(Q) & \lesssim
\Theta_\mu(7R)\sum_{Q\in \treg_{big}(R)} \beta_{\HH^1|_{\Gamma_R},2}(c_3Q)^2\,\mu(Q) \nonumber\\
&\quad + 
\sum_{Q\in \treg_{big}(R)}\sum_{i:P_i\subset c_3Q}
\frac{\ell(\wt P_i)}{\ell(Q)^2}\,\mu(P_i)\,\mu(Q).
\end{align}
For the first sum on the right hand side, using that $\mu(Q)\lesssim\Theta_\mu(7R)\,\ell(Q)$ and that $\Gamma_R$ is an AD-regular curve,
we get
\begin{align*}
\Theta_\mu(7R)\sum_{Q\in \treg_{big}(R)} \beta_{\HH^1|_{\Gamma_R},2}(c_3Q)^2\,\mu(Q) & \lesssim
\Theta_\mu(7R)^2\sum_{Q\in \treg_{big}(R)} \beta_{\HH^1|_{\Gamma_R},2}(c_3Q)^2\,\ell(Q)\\
& \lesssim 
\Theta_\mu(7R)^2\ell(R) \sim \Theta_\mu(7R)\,\mu(7R).
\end{align*}
To estimate the last sum on the right hand side of \rf{eqff104} we use that $\mu(Q)/\ell(Q)\lesssim\Theta_\mu(7R)$ and we
interchange the order the summation:
\begin{align*}
\sum_{Q\in \treg_{big}(R)}\sum_{i:P_i\subset c_3Q}
\frac{\ell(\wt P_i)}{\ell(Q)^2}\,\mu(P_i)\,\mu(Q) & \lesssim \Theta_\mu(7R)
\sum_{i} \mu(P_i)\sum_{Q\in \treg_{big}(R):c_3Q\supset P_i}
\frac{\ell(\wt P_i)}{\ell(Q)}\\
& \lesssim \Theta_\mu(7R)
\sum_{i} \mu(P_i) \lesssim \Theta_\mu(7R)\,\mu(R).
\end{align*}
Hence we get
$$\sum_{Q\in \treg_{big}(R)} \beta_{\mu,2}(3Q)^2\,\mu(Q)  \lesssim \Theta_\mu(7R)\,\mu(7R),$$
which together with \rf{eqff110} gives \rf{eqff90}.
\vv

To conclude it remains to prove the claim \rf{eqclaim39}. 
So let $Q\in\treg_{big}(R)$ and $P_i$ such that $P_i\cap 3Q\neq\varnothing$.
Clearly, the statement in the claim is equivalent to saying that $\ell(\wt P_i)\lesssim\ell(Q)$.
To prove this we distinguish two cases. Suppose first that $P_i\in\treg_{big}(R)$.
 In this case $\ell(\wt P_i)\sim
\ell(P_i)$ since $cP_i\cap\Gamma_R\neq\varnothing$ for some absolute constant $c>1$. 
We may assume that
\begin{equation}\label{eqsupp98}
\ell(Q)<\ell(P_i)/4,
\end{equation}
since otherwise $\ell(\wt P_i)\sim\ell(P_i)\lesssim \ell(Q)$ and we are done.
It is easy to check that the condition \rf{eqsupp98} implies that $Q\subset2P_i$. By the definition of the squares in 
$\treg(R)$ we have $Q\cap\supp\mu\neq\varnothing$, and then from the properties of the family $\reg_*(R)$ in Lemma \ref{lemreg*} we infer that there exists some square
$P_j\in\reg_*(R)$ with $P_j\subset Q$. Since $2P_j\cap 2P_i\neq\varnothing$, we deduce that $\ell(P_j)\sim\ell(P_i)$, by 
(d) in Lemma \ref{lemreg*}. So we infer that
$$\ell(Q)\geq \ell(P_j)\sim\ell(P_i)\sim\ell(\wt P_i),$$
as wished.

Suppose now that $P_i\in\treg_{small}(R)$. Let $S\in\WW(\Omega_R)$ be such that $P_i\subset S$, so that 
$\ell(\wt P_i)\sim\ell(S)$. We have to show that $\ell(S)\lesssim\ell(Q)$. To this end, assume that
$\ell(Q)<\ell(S)/4$, otherwise we are done. This implies that $Q\subset 3S$. 
Since $\Gamma_R$ has empty interior, there exists a Whitney square $S'\in\WW(\Omega_R)$ such that
$S'\cap Q\neq\varnothing$. Since $Q\in\treg_{big}(R)$, we have $S'\subset Q$, and thus $3S'\cap 3S\neq\varnothing$, and then
by the property (b) of Whitney squares, we derive $\ell(S)\sim\ell(S')$. Thus, we get
$$\ell(Q)\geq \ell(S')\sim \ell(S),$$
as desired.
\end{proof}

\vv

Finally, to prove Corollary \ref{corogam} we use Theorem \ref{teocurv-intro} and the main theorem from \cite{Tolsa-sem}, which asserts that, given a compact set $E\subset\C$, we have
$\gamma(E)\sim\mu(E)$,
where the supremum is taken over all measures $\mu$ supported on $E$ such that $\mu(B(x,r))\leq r$ for all $x\in\C$, $r>0$, and $c^2(\mu)\leq \mu(E)$.
Indeed, given $\mu$ satisfying these estimates, by applying Theorem \ref{teocurv-intro} and Chebyshev, we find a compact subset $F\subset E$ 
such that $\mu(F)\geq \mu(E)/2$ and 
$$
\int_0^\infty \beta_{\mu,2}(x,r)^2\,\Theta_\mu^1(x,r)\,\frac{dr}r\lesssim 1 \qquad\mbox{for all $x\in F$}.$$ 
Thus,   for a suitable positive absolute constant $c$, it easily follows that the measure $\nu=c\,\mu|_F$ satisfies 
$$\sup_{r>0}\Theta_\nu^1(x,r) +
\int_0^\infty \beta_{\nu,2}(x,r)^2\,\Theta_\nu^1(x,r)\,\frac{dr}r\leq1 \qquad\mbox{for all $x\in E$},$$
and moreover $\gamma(E)\gtrsim\mu(E)\sim\nu(F)$.
The arguments to prove the converse inequality in Corollary \ref{corogam} are similar.

\vvv


\section{Appendix: Proof of Lemma \ref{angles}} \label{secapp}

\vv
We recall the statement of the lemma.

\begin{lemma*}
Suppose $P_{1}$ and $P_{2}$ are $n$-planes in $\R^{d}$ and $X=\{x_{0},...,x_{n}\}$ are points so that
\begin{enumerate}
\item[(a)] $\eta=\eta(X)=\dfrac1{\diam X}\min\{\dist(x_{i},\spn X\backslash\{x_{i}\})\in (0,1)$ and
\item[(b)] $\dist(x_{i},P_{j})<\ve\,\diam X$ for $i=0,...,n$ and $j=1,2$, where $\ve<\eta d^{-1}/2$.
\end{enumerate}
Then
\begin{equation}
\dist(y,P_{1}) \leq \ve\ps{\frac{2d}{\eta}\dist(y,X)+\diam X}.
\end{equation}
\end{lemma*}

\begin{proof}
Assume first that $x_{0}=0\in P_{1}\cap P_{2}$ and $X\subset P_{1}$. Define a linear map $A:\R^{n}\rightarrow \R^{d}$ by setting $A(e_{i})=x_{i}$, were $e_{1},...,e_{n}$ are the standard basis vectors (but $e_{0}=0$). Then
\[|A|=\sup_{|z|=1}|Az|\leq \sup_{|z|=1}\sum|\ip{z,e_{i}}|\cdot |x_{i}|\leq d^{\frac{1}{2}}|z|\diam X.\]
Let $z\in \R^{n}$ be so that $|A^{-1}|^{-1}=|Az|$ and let $i$ be such that $|\ip{z,e_{i}}|\geq n^{-\frac{1}{2}}|z|$ (since this has to happen for some $\ip{z,e_{i}}$). Then we get
\[
|A^{-1}|^{-1}=|Az|
\geq \dist(Az,\spn(X\backslash\{x_{i}\}))
=|\ip{z,e_{i}}|\dist(x_{i},\spn(X\backslash\{x_{i}\}))
\geq n^{-\frac{1}{2}}|z|\eta\diam X
\]
Thus, we have that $A/\diam X$ is $n^{1/2}/\eta$-bi-Lipschitz. If we define another operator $B$ by setting $B(e_{i})=\pi_{P_{2}}(x_{i})$. Then, for any $z\in\R^{n}$, by our standing assumptions,
\[|A(z)-B(z)|
=\av{\sum_{i=1}^{n}(A-B)(e_{i})\ip{z,e_{i}}}
\leq \ve |z|d^{\frac{1}{2}}\diam X.\]
Hence, since $\ve<\eta d^{-1}/2$,
\begin{align*}
|B'|
& =\sup_{|z|=1}|B(z)|\leq \ve d^{\frac{1}{2}}|z|\diam X+\sup_{|z|=1}|A(z)|
\leq \frac{\eta}{2d^{\frac{1}{2}}} \cdot d^{\frac{1}{2}} |z|\diam X +d^{\frac{1}{2}}|z|\diam X\\
& <2d^{\frac{1}{2}}|z|\diam X\end{align*}
and
\[\inf_{|z|=1}|B(z)|
\geq \inf_{|z|=1}|A(z)|- \ve d^{\frac{1}{2}}\diam X
\geq \frac{1}{2d^{\frac{1}{2}}}|z|\eta\diam X.\]

Thus, $B/\diam X$ is a $\frac{2d^{\frac{1}{2}}}{\eta}$-bi-Lipschitz map from $\R^{d}$ onto $P_{2}$. For $y\in P_{2}$, if $B(z)=y$, then $A(z)\in P_{1}$ and so
\[
\dist(y,P_{1})
 \leq |A(z)-y| 
=|A(z)-B(z)| 
\leq \ve |z|d^{\frac{1}{2}} \diam X \leq \frac{2d\ve }{\eta}|B(z)|=\frac{2d\ve}{\eta}|y|\]
and for $y\in P_{1}$, if $A(z)=y$, then
\[\dist(y,P_{2})
\leq |B(z)-y| 
=|B(z)-A(z)|
\leq \ve |z|d^{\frac{1}{2}}\diam X
\leq \frac{d\ve}{\eta}|y|.\]

Now, if it happens that $0\not\in P_{1}\cap P_{2}$ but $X\subset P_{2}$, we can replace $P_{1}$ with $P_{1}'$, the translate of $P_{1}$ containing $0$, and apply the same arguments above to get an estimate between $P_{1}'$ and $P_{2}$. Since $P_{1}$ and $P_{1}'$ are distance at most $\dist(0,P_{1})<\ve\,\diam X$, this gives
\[\dist (y,P_{1})\leq \ve \,\ps{\frac{2d}{\eta}|y|+\diam X} \mbox{ for all }y\in P_{2}\]
and
\[\dist (y,P_{2})\leq \ve \,\ps{\frac{2d}{\eta}|y|+\diam X} \mbox{ for all }y\in P_{1}.\]
Now, if $X\not\subset P_{2}$, let $P_{0}$ denote the smallest $n$-plane containing $X$ (again, assume $x_{0}=0$). Then we apply the above estimates between $P_{1}$ to $P_{0}$ and $P_{0}$ to $P_{2}$ using the triangle inequality and we obtain
\[\dist (y,P_{1})\leq 2\ve \ps{\frac{2d}{\eta}|y|+\diam X} \mbox{ for all }y\in P_{2}\]
and
\[\dist (y,P_{2})\leq 2\ve \ps{\frac{2d}{\eta}|y|+\diam X} \mbox{ for all }y\in P_{1}.\]

Now, there is no need to assume $x_{0}=0$, since we can just replace $|y|$ with $|x_{0}-y|$ above. Finally, there was no special reason we dealt with $x_{0}$ in particular, and so minimizing the above inequalities over all $|x_{0}-y|,....,|x_{n}-y|$, we obtain the desired estimate.
\end{proof}

\vv

\end{document}